\documentclass[12pt, twoside, reqno]{amsart}
\usepackage{amscd}
\usepackage{bbm}
\usepackage{dsfont}
\usepackage{enumerate}
\usepackage{latexsym}
\usepackage{srcltx}
\usepackage{amsfonts}
\usepackage{amssymb}
\usepackage{datetime}
\usepackage{setspace}
\usepackage{enumerate}
\usepackage{amsthm}
\usepackage{graphicx}
\usepackage{epstopdf}
\usepackage{amsmath}
\usepackage{fancyhdr}

\newtheorem{theorem}{Theorem}[section]
\newtheorem{proposition}[theorem]{Proposition}
\newtheorem{lemma}[theorem]{Lemma}
\newtheorem{corollary}[theorem]{Corollary}

\theoremstyle{definition}
\newtheorem{example}[theorem]{Example}
\newtheorem{definition}[theorem]{Definition}

\newtheorem{remark} [theorem] {Remark}
\newtheorem{conjecture} [theorem] {Conjecture}

\input epsf
\begin{document}
\title{ Spectrum of Weighted Composition Operators \\
Part VI \\
Essential spectra of $d$-endomorphisms of Banach $C(K)$-modules }

\author{Arkady Kitover}

\address{Community College of Philadelphia, 1700 Spring Garden St., Philadelphia, PA,
USA}

\email{akitover@ccp.edu}

\author{Mehmet Orhon}

\address{University of New Hampshire, Durham, NH, 03824}

\email{mo@unh.edu}


\date{\today}

\centerline{Dedicated to the memory of Yuri Abramovich}
\begin{abstract}
  We investigate properties of essential spectra of disjointness preserving operators acting on Banach $C(K)$-modules. In particular, we prove that under some very mild conditions the upper semi-Fredholm spectrum of such an operator is rotation invariant. In the last part of the paper we provide a full description of the spectrum and the essential spectra of operators acting on Kaplansky modules of the form $T = wU$, where $w \in C(K)$, $U$ is a $d$-isomorphism, and the spectrum of $U$ is a subset of the unit circle.
\end{abstract}
\maketitle
\markboth{Arkady Kitover and Mehmet Orhon}{Spectrum of weighted composition operators.
VI}

\section{introduction}

Disjointness preserving operators on Banach $C(K)$-modules \footnote{The reader is referred to the next section for the precise meaning of terms and notions used in the introduction.} were introduced in~\cite{AAK}. One of the main results proved there stated that if the powers of such an operator $T$ are in some sense disjoint and every cyclic subspace of the Banach $C(K)$-module (represented as a Banach lattice) has the Fatou property, then the spectrum and approximate point spectrum of $T$ are rotation invariant. It is natural to ask whether under the same conditions the essential spectra (e.g. the Fredholm spectrum) of $T$ would be rotation invariant as well. At the time of publication of~\cite{AAK}, almost thirty years ago, we knew too little about the properties of Banach $C(K)$-modules and about essential spectra of disjointness preserving operators on Banach lattices to tackle this problem. During the last few years the situation changed for the better and, though many questions about essential spectra remain open, we were able to make some progress presented in the current paper. The plan of the paper is as follows.

 \noindent \textbf{Section 2} contains notations, basic definitions, and some statements from previous publications included for the reader's convenience.

\noindent \textbf{Section 3}. We show in Theorem~\ref{t9} that if the powers of $T$ are "disjoint" ($d$-independent) and all cyclic subspaces have the Fatou property then the upper semi-Fredholm spectrum is rotation invariant.

If we assume a much stronger condition that all cyclic subspaces have order continuous norm, then all the essential spectra of $T$ considered in this paper are rotation invariant (see Corollary~\ref{c1}).

\noindent \textbf{Section 4}. In this section we consider the case when the conjugate operator $T'$ is also disjointness preserving. This assumption allows us to obtain stronger results than in the previous section.

\noindent \textbf{Section 5}. In this section we consider operators of the form $T = wU$ where $w \in C(K)$, $U$ is a $d$-isomorphism, and the spectrum of $U$ lies on the unit circle. These assumptions allow us to obtain a complete description of essential spectra of $T$.

\noindent \textbf{Section 6}. The section contains some examples.

\noindent \textbf{Appendix}. In the appendix we carefully reexamine some notions connected with the disjointness preserving operators on complex Banach lattices and Banach $C(K)$-modules. In particular, we show in Theorem~\ref{at6} that the (rather technical) definition of disjointness given in~\cite{AAK} is equivalent to a much simpler and more natural definition.

\section{Preliminaries}

\subsection{Basic notations and definitions}

In the sequel we use the following standard notations.

\noindent $\mathds{N}$ is the semigroup of all natural numbers.

\noindent $\mathds{Z}$ is the ring of all integers.

\noindent $\mathds{R}$ is the field of all real numbers.

\noindent $\mathds{C}$ is the field of all complex numbers.

\noindent $\mathds{T}$ is the unit circle. We use the same notation for the unit circle
considered as a subset of the complex plane and as the group of all complex numbers
of modulus 1.

\noindent $\mathds{U}$ is the open unit disc.

\noindent $\mathds{D}$ is the closed unit disc.

All the linear spaces are considered over the field $\mathds{C}$ of complex numbers.

The algebra of all bounded linear operators on a Banach space $X$ is denoted by $L(X)$.

Let $E$ be a set, $\varphi : E \rightarrow E$ be a map, and $w$ be a
complex-valued function on $E$. Then

\noindent $\varphi^n$ , $n \in \mathds{N}$, is the $n^{th}$ iteration of $\varphi$,

\noindent $\varphi^0(e) = e, e \in E$,

\noindent If $F \subseteq E$ then $\varphi^{(-n)}(F)$ means the full preimage of $F$ for the map $\varphi^n$.

\noindent If the map $\varphi$ is invertible then $\varphi^{-n}$, $n \in \mathds{N}$,  is the $n^{th}$ iteration of the inverse map $\varphi^{-1}$.

\noindent $w_0 = 1$, $w_n = w(w \circ \varphi) \ldots (w \circ \varphi^{n-1})$, $n \in
\mathds{N}$.

Recall that an operator $T \in L(X)$ is called \textit{semi-Fredholm} if its range
$R(T)$ is closed in $X$ and either $\dim{\ker{T}}< \infty$ or codim $R(T) < \infty$.

The \textit{index} of a semi-Fredholm operator $T$ is defined as

\centerline{ ind $T$ = $  \dim{\ker{T}}$ - $\mathrm{codim} \, R(T)$.}

The subset of $L(X)$ consisting of all semi-Fredholm operators is denoted by $\Phi$.

$\Phi_+ = \{T \in \Phi : null(T) = \dim{\ker{T}}< \infty\}$ is the set of all upper semi-Fredholm operators in $L(X)$.

$\Phi_- = \{T \in \Phi : def(T) = \textrm{codim}\; R(T) < \infty\}$ is the set of all lower semi-Fredholm operators in $L(X)$.

$\mathcal{F} = \Phi_+  \cap  \Phi_-$ is the set of all Fredholm operators in $L(X)$.

$\mathcal{W} = \{T \in \mathcal{F} : \textrm{ind} \; T = 0\}$ is the set of all Weyl
operators in $L(X)$.

Let $T$ be a bounded linear operator on a Banach space $X$. As usual, we denote the
spectrum of $T$ by $\sigma(T)$ and its spectral radius by $\rho(T)$.

We will consider the following subsets of $\sigma(T)$.

$\sigma_p(T) = \{\lambda \in \mathds{C} : \exists x \in X \setminus \{0\}, Tx = \lambda
x\}.$

$\sigma_{a.p.}(T) = \{\lambda \in \mathds{C}: \exists x_n \in X, \|x_n\| = 1, Tx_n -
\lambda x_n \rightarrow 0\}$.

$\sigma_r(T) = \sigma(T) \setminus \sigma_{a.p.}(T) =$

$\; \; = \{\lambda \in \sigma(T) : \textrm{the operator}\; \lambda I - T \; \textrm{has left inverse}\} $.

\begin{remark} \label{r1} It is clear that $\sigma_{a.p.}(T)$ is the union of the point spectrum $\sigma_p(T)$ and the approximate point spectrum $\sigma_a(T)$ of $T$, while $\sigma_r(T)$ is the residual spectrum of $T$. We have to notice that the  definition of the residual spectrum varies in the literature.
\end{remark}

Following~\cite{EE} we consider the following essential spectra of $T$.

$\sigma_1(T) = \{\lambda \in \mathds{C}: \lambda I - T \not \in \Phi\}$ is the
\textit{semi-Fredholm} spectrum of $T$.

$\sigma_2(T) = \{\lambda \in \mathds{C}: \lambda I - T \not \in \Phi_+\}$ is the upper \textit{semi-Fredholm} spectrum of $T$.

$\sigma_2(T^\prime)  = \{\lambda \in \mathds{C}: \lambda I - T \not \in \Phi_-\} $ is the lower \textit{semi-Fredholm} spectrum of $T$.

$\sigma_3(T) = \{\lambda \in \mathds{C}: \lambda I - T \not \in \mathcal{F}\}$ is the
Fredholm spectrum of $T$.

$\sigma_4(T) = \{\lambda \in \mathds{C}: \lambda I - T \not \in \mathcal{W}\}$ is the
Weyl spectrum of $T$.

$\sigma_5(T) = \sigma(T)\setminus \{\zeta \in \mathds{C} :$ there is a component $C$ of
the set $\mathds{C} \setminus \sigma_1(T)$ such that $\zeta \in C$ and the intersection
of $C$ with the resolvent set of $T$ is not empty$\}$ is the Browder spectrum of $T$.

The Browder spectrum was introduced in~\cite{Br} as follows: $\lambda \in \sigma(T) \setminus \sigma_5(T)$ if and only if $\lambda$ is a pole of the resolvent $R(\lambda, T)$. It is not difficult to see (~\cite[p.40]{EE}) that the definition of $\sigma_5(T)$ cited above is equivalent to the original definition of Browder.

It is well known (see e.g.~\cite{EE}) that the sets $\sigma_i(T ), i \in [1, \ldots ,
5]$ are
nonempty closed subsets of $\sigma(T)$ and that
\begin{equation*}
  \sigma_i(T) \subseteq \sigma_j(T), 1 \leq i < j \leq 5,
\end{equation*}
where all the inclusions can be proper. Nevertheless all the
spectral radii $\rho_i(T ), i = 1, . . . , 5$ are equal to the same number,
$\rho_e(T)$, (see~\cite[Theorem I.4.10]{EE}) which is called the essential
spectral radius of $T$. It is also known (see~\cite{EE}) that the spectra $\sigma_i(T),
i = 1, \ldots , 4$ are invariant under compact perturbations, but $\sigma_5(T)$ in
general is not.

It is immediate to see that $\sigma_1(T) = \sigma_2(T) \cap \sigma_2(T^\prime)$ and that $\sigma_3(T) = \sigma_2(T) \cup \sigma_2(T^\prime)$.

Let us recall that a sequence $x_n$ of elements of a Banach space $X$ is called \emph{singular} if it does not contain any norm convergent subsequence. We will use the following well known characterization of $\sigma_2(T)$ (see e.g.~\cite{EE}). The following statements are equivalent
\begin{enumerate}[(a)]
  \item $\lambda \in \sigma_2(T)$.
  \item There is a singular sequence $x_n$ such that $\|x_n\| = 1$ and $\lambda x_n - Tx_n \rightarrow 0$.
\end{enumerate}

\subsection{Banach $C(K)$-modules and $d$-endomorphisms}

In this subsection we mostly use terminology introduced in~\cite{AAK}, \cite{HO}, and~\cite{KO}.

Let $K$ be a compact Hausdorff space, $X$ be a Banach space, and $m$ be a bounded unital homomorphism of $C(K)$ into $L(X)$.  The triple $(C(K), m, X)$ is called a Banach $C(K)$-module. For brevity we will say that $X$ is a Banach $C(K)$-module.

Throughout this paper we will always assume without sprecifically mentioning it that
\begin{enumerate}[(a)]
\item The Banach $C(K)$-module $X$ is exact, i.e. $\ker{m} = 0$.
Indeed, otherwise we can consider $X$ as a Banach $C(K_1)$-module, where $C(K_1)\simeq C(K)/\ker{m}$.
  \item The homomorphism $m$ is an isometry (under condition $(a)$ above we can assume it without loss of generality, see~\cite[Lemma 2]{HO}).
  \end{enumerate}

\begin{remark} \label{r2} In~\cite{AAK} a Banach $C(K)$-module with properties $(a)$ and $(b)$ is called an \textbf{operator} $C(K)$-module.
  \end{remark}
Let $X$ be a Banach $C(K)$-module. Notice that in general the weak operator closure of $m(C(K))$ is isometrically isomorphic to some $C(Q)$ (see~\cite[Corollary 6.3, p.35]{AAK}). By analogy with the case of Banach lattices this closure is called the center of the Banach $C(K)$-module $X$ (see~\cite[Definition 4.3]{AAK}) but in the current paper we will not use this terminology.

In the sequel when considering a Banach $C(K)$-module we will assume, unless otherwise specified, that

$(c)$ $m(C(K))$ is closed in $L(X)$ in the weak operator topology.

For brevity, throughout the paper we will identify an $f \in C(K)$ and its image $m(f) \in L(X)$.

Let $x \in X$. The \textbf{cyclic subspace} $X(x)$ is defined as
\begin{equation*}
  X(x) = cl_X \{fx : f \in C(K) \}.
\end{equation*}

Endowed with the order defined by the cone $X_+(x) = cl \{fx : f \in C(K), f \geq 0\}$ the cyclic subspace $X(x)$ becomes a Banach lattice (see~\cite[Lemma 2]{OO} and~\cite[Lemma 2.4]{KO}).

The following important class of Banach $C(K)$-modules was introduced by Kaplansky in~\cite{Ka}.

\begin{definition} \label{d1} A Banach $C(K)$-module $X$
   is called a \textbf{Kaplansky} module if

   (1) The compact space $K$ is extremally disconnected.

   (2) For any $x \in X$ the set $\{f \in C(K): fx = 0\}$ is a band in $C(K)$.
   \end{definition}

 \begin{remark} \label{r3} In~\cite{AAK} a Banach $C(K)$-module with properties $(1)$ and $(2)$ from Definition~\ref{d1} is called an \textbf{order complete} $C(K)$-module.
 \end{remark}

 For us, the following properties of Kaplansky modules will be important (for a proof see~\cite{KO}).

\begin{proposition} \label{p10} Let $X$ be a Kaplansky module. Then
\begin{enumerate}
  \item $C(K)$ is closed in $L(X)$ in the weak operator topology.
  \item For every $x \in X$ the cyclic subspace $X(x)$ with the cone $X_+(x)$ is a Dedekind complete Banach lattice.
  \item For every $x \in X$ there exists the unique idempotent $e_x$ in $C(K)$ such that $e_x x = x$ and
\begin{equation*}
  h \in C(K), h^2 = h, hx=x \Rightarrow e_x \leq h.
\end{equation*}
\end{enumerate}
The idempotent $e_x$ is called the \textbf{the carrier projection} of $x$ and by the support of $x$, $supp \; x$, we mean the support of $e_x$ in $K$.
\end{proposition}

  Disjointness preserving operators on Banach $C(K)$-modules were introduced in~\cite[Definition 4.5, p.21]{AAK}. The definition in~\cite{AAK} is technically complicated and not very convenient to work with. Here we will use an equivalent but simpler definition. The proof of equivalence of Definition~\ref{d2} below and Definition 4.5 in~\cite{AAK} is given in Theorem~\ref{at6} in the appendix.

   \begin{definition} \label{d2} Let $X$ be a Banach $C(K)$-module and $T \in L(X)$. The operator $T$ is called a $\mathbf{d}$-\textbf{endomorphism} if for any $x \in X$ we have $T(X(x)) \subseteq X(Tx)$ and the operator $T|X(x) : X(x) \rightarrow X(Tx)$ is disjointness preserving (we assume here that the cyclic subspaces $X(x)$ and $X(Tx)$ are represented as Banach lattices).
   \end{definition}

In the case when $X$ is a Kaplansky module its $d$-endomorphisms can be characterized as follows.

   \begin{proposition} \label{p3}  (see~\cite[Proposition 11.23, p.83]{AAK}) Let $X$ be a Kaplansky module and $T \in L(X)$. The following conditions are equivalent.
\begin{enumerate}
 \item $T$ is a $d$-homomorphism of $X$.
\item There exist a clopen subset $E$ of $K$ and a continuous map $\varphi : E \rightarrow K$ such that
   \begin{equation*}
     T(fx) = (f \circ \varphi) Tx, x \in X, f \in C(K),
   \end{equation*}
where
\begin{equation*}
(f \circ \varphi)(k) =  \left\{
    \begin{array}{ll}
      f(\varphi(k)), & \hbox{if $k \in E$;} \\
      0, & \hbox{if $k \in K \setminus E$.}
    \end{array}
  \right.
\end{equation*}
\end{enumerate}
   \end{proposition}

   \begin{remark} \label{r16} The set $E$ in Proposition~\ref{p3} is the smallest (by inclusion) clopen subset of $K$ such that $\chi_E T = T$. Respectively, $F = K \setminus E $ is the largest clopen subset of $K$ such that $\chi_F T = 0$. In particular, if $E = K$ and $F = \emptyset$ the map $\varphi$ is defined on $K$.
   \end{remark}

   \begin{remark} \label{r21} Because the map $\varphi$ in the statement of Proposition~\ref{p3} is defined only on a clopen subset of $K$ we need to clarify what we mean by iterations of the map $\varphi$ and its periodic points. For any $n \in \mathds{N}$ let $E_n$ and $\varphi_n : E_n \rightarrow K$ be the clopen subset of $K$ and the map corresponding to the $d$-endomorphism $T^n$, respectively. Thus, $E = E_1$ and $\varphi = \varphi_1$. It is immediate to see that $E_n \subseteq \varphi^{-1}(E_{n-1}) \cap E_{n-1}, \; n \geq 1$. \footnote{The inclusion
    $E_n \subseteq \varphi^{-1}(E_{n-1}) \cap E_{n-1}$ can be proper.} Therefore, the map $\varphi^n$ is well defined on $E_n$ and $\varphi^n(k) = \varphi_n(k), k \in E_n$.
   \end{remark}

   \begin{definition} \label{d11} Let $E$ and $\varphi$ be as in the statement of Proposition~\ref{p3}. A point $k \in K$ is called $\mathbb{\varphi}$-\textbf{periodic} if $k \in \bigcap \limits_{n=1}^\infty E_n$ and there is an $m \in \mathds{N}$ such that $\varphi^{m}(k) = k$. The smallest such $m$ is called the \textbf{period} of $k$.

   A point $k \in K$ is called \textbf{eventually} $\mathbb{\varphi}$-\textbf{periodic} if $k \in \bigcap \limits_{n=1}^\infty E_n$ and there is a $p \in \mathds{N}$ such that $\varphi^{p}(k)$ is a $\varphi$-periodic point.
   \end{definition}

Now we need to discuss the notion of $\mathbf{d}$-\textbf{independence} of powers of a $d$-endomorphism introduced in~\cite{AAK}. First recall that if $X$ is a Dedekind complete Banach lattice then the space $L_r(X)$ of all regular operators on $X$ is a Banach lattice as well. In this case we say that powers $T^n : n \geq 0$ of a $d$-endomorphism $T$ are $d$-independent if these powers are disjoint in the Banach lattice $L_r(X)$.

\begin{proposition} \label{p4}(see~\cite[Proposition 11.24 and Remark 11.25, p.83]{AAK})  Let $X$ be a Dedekind complete Banach lattice, $Z(X) = C(K)$ be the center of $X$, $T$ be a $d$-endomorphism of $X$, and $\varphi$ be the corresponding map described in Proposition~\ref{p3}. For $m \in \mathds{N}$ let $F^{(m)}$ be the set of all $\varphi$-periodic points in $K$ of period less or equal to $m$. The following conditions are equivalent
\begin{enumerate}
  \item The powers of $T$ are $d$-independent.
  \item For any $m, n \in \mathds{N}$ the set $\varphi^{(-n)}(F^{(m)})$ is nowhere dense in $K$.
\end{enumerate}
\end{proposition}

Proposition~\ref{p4} justifies the following definition.

   \begin{definition} \label{d3} Let $X$ be a Kaplansky module and $T$ be a $d$-endomorphism of $X$. We say that the powers $T^n : n \geq 0$ of $T$ are $d$-\textit{independent} if for any $m, n \in N$ the set
   $\varphi^{(-n)}(F^{(m)})$ is nowhere dense in $K$.
   \end{definition}

To be able to define $d$-independence of powers of a $d$-endomorphism in the case of an arbitrary Banach $C(K)$-module $X$ we need to discuss briefly the properties of the conjugate module $X'$. The map $f \rightarrow f', f \in C(K)$ defines an isometrical and algebraic embedding of $C(K)$ into $L(X')$. Let us consider the closure of the image of $C(K)$ under this map in the topology $\sigma(L(X'), X' \otimes X)$. It was proved in~\cite[Chapter 9]{AAK} that this closure is isometrically isomorphic to $C(K_1)$ where $K_1$ is a hyperstonean compact space, and that $X'$ considered as a Banach $C(K_1)$-module is a Kaplansky module.

Applying the same procedure to the Banach $C(K_1)$-module $X'$ we see that the second dual $X^{\prime \prime}$ can be endowed with the structure of a Kaplansky $C(Q)$-module where $Q$ is a hyperstonean compact space. Moreover, if $T$ is a $d$-endomorphism of $X$ then $T^{\prime \prime}$ is a $d$-endomorphism of $X^{\prime \prime}$ (see~\cite[Theorem 10.6, p.71 and Remark 11.22.4, p.83]{AAK}). By Proposition~\ref{p3} there are a clopen subset $E$ of $Q$ and a continuous map
$\psi$ such that
\begin{equation} \label{eq54}
  T^{\prime \prime}gz = (g \circ \psi)T^{\prime \prime}z, g \in C(Q), z \in X^{\prime \prime}
\end{equation}
where the composition $g \circ \psi$ is defined as in Proposition~\ref{p3}. Moreover, the map $\psi$ is open (\cite[Theorem 10.7, p.71]{AAK}).

Let $x \in X$ and let $X(x)$ be the corresponding cyclic subspace endowed with the structure of a Banach lattice. Let $Z(X)$ be the center of the Banach lattice $X(x)$ and $K_x$ be the Gelfand space of $Z(X)$. Let also $j$ be the canonical map of $X$ into $X^{\prime \prime}$. Let $Q_x$ be the support of $jx$ in $Q$. Then (see~\cite[11.35, p.86]{AAK}) there exists a continuous unique surjective map $\eta_x : Q_x \rightarrow K_x$ such that
\begin{equation*}
  jTfx = (f \circ \eta_x)jTx, f \in Z(X(x)).
\end{equation*}

\begin{definition} \label{d6} Let $X$ be a Banach $C(K)$-module, $T$ be a

\noindent $d$-endomorphism of $X$, and $\psi$ be as in (\ref{eq54}). Let $F^{(m)}, m \in \mathds{N}$, be the set of all $\psi$ periodic points in $Q$  of period at most $m$. We say that powers of $T$ are $\mathbf{d}$-\textbf{independent} if for any $x \in X$ and for any $m, n \in \mathds{N}$ the set $\eta_x(\psi^{(-n)}(F^{(m)} \cap Q_x))$ is nowhere dense in $K_x$.
\end{definition}

   \begin{remark} \label{r10} The definition of $d$-independence in~\cite[Definition 11.8]{AAK} is technically much more involved than Definition~\ref{d6}, but as we will prove in the appendix (see Theorem~\ref{at6}) the two definitions are equivalent. In particular, in the case of Kaplansky modules Definition~\ref{d6} is equivalent to Definition~\ref{d3}.
   \end{remark}

\subsection{Some results on spectrum of $d$-endomorphisms.} In this subsection we recall some results concerning the spectrum of $d$-endomorphisms of Banach $C(K)$-modules and Banach lattices that we will need in this paper.

\bigskip

   Let us recall that a Banach lattice $X$ has the Fatou property if for every positive $x \in X$ and for every net $x_\alpha$ such that $x_\alpha \geq 0$ and $x_\alpha \uparrow x$ we have $\|x_\alpha\| \uparrow \|x\|$.

   One of the main results in~\cite{AAK} is the following theorem.

   \begin{theorem} \label{t20} (~\cite[Theorem 12.11]{AAK} Let $X$ be a Banach $C(K)$-module and $T$ be a $d$-endomorphism of $X$. Assume the following two conditions.

  \noindent (1) Every cyclic subspace of $X$ represented as a Banach lattice has the Fatou property.

  \noindent (2) The powers of $T$ are independent.

  Then the sets $\sigma_{a.p.}(T)$ and $\sigma(T)$ are rotation invariant.
   \end{theorem}

In the case when $X$ is a Kaplansky module we can say more. To discuss it we need some additional notations.

\begin{definition} \label{d9}  Let $X$ be a Kaplansky module, $T$ be a $d$-endomorphism of $X$ and $\varphi$ be the corresponding map (see Proposition~\ref{p3}). For every $m \in \mathds{N}$ we denote by $P_m$ the set of all $\varphi$-periodic points of the period $m$, and by $\Pi_m$ the clopen set $Int P_m$. We also denote by $H_m$ the clopen set $cl \bigcup \limits_{n=1}^\infty \varphi^{(-n)}(\Pi_m)$.
\end{definition}

\begin{remark} \label{r19} In general $H_m \subsetneqq \varphi^{(-1)}(H_m)$, but if the map $\varphi$ is open then it is easy to see that $H_m = \varphi^{(-1)}(H_m)$. In connection with it we recall the following definition from~\cite{AAK}.
\end{remark}

\begin{definition} \label{d7} Let $X$ be a Banach $C(K)$-module and $T$ be a $d$-endomorphism of $X$. The operator $T$ is called \textbf{order continuous} if for every $x \in X$ the operator $T : X(x) \rightarrow X(Tx)$ (where $X(x)$ and $X(Tx)$ are represented as Banach lattices) is order continuous.
\end{definition}

\begin{proposition} \label{p11} (see~\cite[Remark 11.28, p.84]{AAK}) Let $X$ be a Kaplansky module and $T$ be a $d$-endomorphism of $X$. The following conditions are equivalent.
  \begin{enumerate}
    \item The operator $T$ is order continuous.
    \item The map $\varphi$ from Proposition~\ref{p3} is open.
  \end{enumerate}
\end{proposition}

\begin{definition}  \label{d8} Let $X$ be a Kaplansky module, $T$ be an order continuous $d$-endomorphism of $X$, and $\varphi$ be the corresponding map from Proposition~\ref{p3}. Let $\chi_m$ be the characteristic function of the set $\Pi_m$ and $\chi_{(m)}$ - the characteristic function of $H_m$.  We denote by $T_m$ and $T_{(m)}$ the operators $\chi_m T \chi_m$ and $\chi_{(m)}T)$, respectively. We also denote by $\sigma_m$ and $\sigma_{(m)}$ the sets $\sigma(\chi_m T \chi_m)$ and $\sigma(\chi_{(m)}T)$, respectively.
\end{definition}

Now we can state the result that complements Theorem~\ref{t20}.

\begin{theorem} \label{t22} (see~\cite[Proposition 12.15 and Remark 12.16]{AAK}) Let $X$ be a Kaplansky module and $T$ be an order continuous $d$-endomorphism of $X$. Assume one of the following conditions
\begin{enumerate}[(a)]
  \item $X$ has the Fatou property.
  \item $K$ is hyperstonean.
\end{enumerate}
Then
\begin{enumerate}
  \item $\gamma \sigma_m = \sigma_m$ and $\gamma \sigma_{(m)} = \sigma_{(m)}$ for every $\gamma \in \mathds{C}$ such that $\gamma^m = 1$.
  \item For every $m \in \mathds{N}$ we have $\sigma_{(m)} \subseteq \sigma(T)$ and the set
$\sigma(T) \setminus \bigcup \limits_{m=1}^\infty \sigma_{(m)}$ is rotation invariant.
  \item For every $m \in \mathds{N}$ the set $\sigma_{a.p.}(\chi_{(m)}T) \setminus \sigma_{a.p.}(\chi_m T \chi_m)$ is rotation invariant.
\end{enumerate}
\end{theorem}

There are two important cases when the statement of Theorem~\ref{t22} can be further improved. The first of them is the case when the Kaplansky module is actually a Dedekind complete Banach lattice.

\begin{theorem} \label{t23} Let $X$ be a Dedekind complete Banach lattice with the Fatou property and $T$ be an order continuous $d$-endomorphism of $X$. Then
\begin{enumerate}
  \item $\sigma_m \subseteq \sigma(T), m \in \mathds{N}$.
  \item The set $\sigma(T) \setminus \bigcup \limits_{m=1}^\infty \sigma_m$ is rotation invariant.
\end{enumerate}
\end{theorem}

\begin{remark} \label{r12} The inclusion $\sigma_m \subseteq \sigma(T)$ or even $\sigma_m \subseteq \sigma(T) \cup \{0\}$ is not true in general in the case when $X$ is a Kaplansky module (see~\cite[Remark 12.17]{AAK}).
  \end{remark}

The second is the case when the conjugate operator $T'$ is a $d$-endomorphism of the Kaplansky module $X'$.

\begin{theorem} \label{t32} Let $X$ be a Kaplansky module and $T$ be a $d$-endomorphism of $X$ such that $T'$ is a $d$-endomorphism of $X'$. Assume one of the following conditions.
\begin{enumerate}
  \item $X$ has the Fatou property.
  \item $T$ is order continuous.
\end{enumerate}
Then $\sigma_m = \sigma_{(m)}, m \in \mathds{N}$, and the set $\sigma(T) \setminus \bigcup \limits_{m=1}^\infty \sigma_m$ is rotation invariant.
\end{theorem}

\begin{remark} \label{r13} Condition (2) of Theorem~\ref{t32} is satisfied, in particular, if $U$ is a $d$-automorphism of $X$, $w \in C(K)$, and $T = wU$.
\end{remark}

At the end of this subsection we will recall some results about essential spectra of weighted automorphisms of $C(K)$ and operators of the form $T = wU$ acting on a Dedekind complete Banach lattice $X$, where $U$ is a $d$-automorphism of $X$, $\sigma(U) \subseteq \mathds{T}$, and $w \in Z(X)$. We will need these results later in Section 5.

The next two theorems were proved in~\cite{Ki1}. We will need them only in the case when the compact space $K$ is extremally disconnected (stonean), thus their statements can be somewhat simplified.

\begin{theorem} \label{t35} (~\cite[Theorem 3.29]{Ki1}) Let $\varphi$ be a homeomorphism of the compact stonean space $K$ onto itself, $w \in C(K)$, and $(Tf)(k) = w(k)f(\varphi(k), f \in C(K), k \in K$. Assume that the set  of all $\varphi$-periodic points is empty. Let $\lambda \in \sigma(T)$. Consider the following statements.

 \noindent $(R)$ The operator $\lambda I - T$ has a right inverse, or equivalently $(\lambda I - T)C(K) = C(K)$, or equivalently $\lambda \in \sigma_r(T^\prime)$.

 \noindent $(L)$  The operator $\lambda I - T$ has a left inverse, or equivalently

\noindent $\|(\lambda I -T)f\| \geq C\|f\|, \; f \in C(K), C > 0$, or equivalently $\lambda \in \sigma_r(T)$.

 \noindent $(A)$ There are closed subsets $E, F$ and $Q$ of $K$ such that
\begin{enumerate}[(a)]
  \item The set $Q$ is clopen and the sets $E, F$, and $\varphi^i(Q), i \in \mathds{Z}$ are pairwise disjoint,
\item $K = E \cup F \cup \bigcup \limits_{j=-\infty}^\infty \varphi^j(Q)$.
  \item $\varphi(E) = E$ and $\sigma(T, C(E))\subset \{\xi \in \mathds{C} : |\xi| < |\lambda|\}$,
  \item $\varphi(F) = F$ and $\sigma(T, C(F))\subset \{\xi \in \mathds{C} : |\xi| > |\lambda|\}$,
  \item $\bigcap \limits_{n=1}^\infty  cl \bigcup \limits_{j=n}^\infty \varphi^j(Q) \subseteq E$,
  \item $\bigcap \limits_{n=1}^\infty cl \bigcup \limits_{j=n}^\infty \varphi^{-j}(Q) \subseteq F$.
 \end{enumerate}
\noindent $(B)$  There are closed subsets $E, F$ and $Q$ of $K$ with the properties (a) - (d) from $(A)$ and such that

\begin{equation*} (e') \bigcap \limits_{n=1}^\infty  cl \bigcup \limits_{j=n}^\infty \varphi^j(Q) \subseteq F
\end{equation*}
 and
\begin{equation*}  (f') \bigcap \limits_{n=1}^\infty cl \bigcup \limits_{j=n}^\infty \varphi^{-j}(Q) \subseteq E.
\end{equation*}
Then the following equivalencies hold:
\begin{equation*}
  (L) \Leftrightarrow (A),
\end{equation*}
\begin{equation*}
  (R) \Leftrightarrow (B).
\end{equation*}
\end{theorem}

The next theorem indicates the changes we have to make in the statement of Theorem~\ref{t35} if we omit the condition that the set of all $\varphi$-periodic points is empty.

\begin{theorem} \label{t33} (~\cite[Theorem 3.31]{Ki1}) Let $\varphi$ be a homeomorphism of the compact stonean space $K$ onto itself, $M \in C(K)$, and $(Tf)(k) = M(k)f(\varphi(k), f \in C(K), k \in K$. Let $\lambda \in \sigma(T)$. The following conditions are equivalent.

\noindent $(1)$ $\lambda \in \sigma_r(T)$ (respectively, $\lambda \in \sigma_r(T'))$.

\noindent $(2)$ There are $m \in \mathds{N}$ and a clopen subset $P$ of $K$ such that  $P \subset P_m$, $\varphi(P) = P$, $\lambda \not \in \sigma(T, C({cl P}))$, and $K$ can be partitioned as $K = E \cup Q \cup F \cup P$ where the sets
$E,F$, and $Q$ satisfy conditions $A$ (respectively $B$) of Theorem~\ref{t32}.
\end{theorem}

Thus we have a description of the conditions when the operator $\lambda I - T$ is semi-Fredholm and either $null T = 0$ or $ def T =0$.  The next theorem combined with Theorem~\ref{t35} provides a description of the upper semi-Fredholm spectrum of a weighted automorphism of $C(K)$.

\begin{theorem} \label{t34} (see~\cite[Theorem 2.7]{Ki3}) Let $T$ be a weighted automorphism of $C(K)$,
\begin{equation*}
  (Tf)(k) = w(k)f(\varphi(k)), f \in C(K), k \in K,
\end{equation*}
 and let $\lambda \in \sigma(T) \setminus \{0\}$. The following conditions are equivalent.

(I) The operator $\lambda I - T$ is semi-Fredholm and $0 < nul(\lambda I - T) < \infty$.

(II) There are subsets $L$ and $S$ of $K$ with the properties

\begin{enumerate}
\item The subsets $L$ and $S$ are at most finite and at least one of them is not empty.
  \item Every point of $L \cup S$ is an isolated point in $K$.
\item If $p, q \in L \cup S$ and $p \neq q$ then $\varphi^i(p) \neq \varphi^j(q), i, j \in \mathds{Z}$
  \item If $l \in L$ then the point $l$ not $\varphi$-periodic and $\lambda \in \sigma_r(T', C(tr(l))')$, where $Tr(l) = cl \{\varphi^j(l) ; j \in \mathds{Z}$.
  \item If $s \in S$ then the point $s$ is  $\varphi$-periodic and $\lambda^p = w_p(s)$, where $p$ is the period of the point $s$.
    \item Let $U = \bigcup \limits_{j= - \infty}^\infty \varphi^{(j)}(L \cup S)$ and $V = K \setminus U$. Then either $\lambda \not \in \sigma(T, C(V))$ or  $\lambda \in \sigma_r(T, C(V))$.
\end{enumerate}

\end{theorem}

 \begin{remark} \label{r14}. Assume conditions of Theorem~\ref{t34}. Then $null(\lambda I - T) = card(L) + card(S)$.
 \end{remark}

Similarly we can describe the lower semi-Fredholm spectrum of a weighted automorphism of $C(K)$.

\begin{theorem} \label{t36} (see~\cite[Theorem 2.11]{Ki3}) Let $T$ be a weighted automorphism of $C(K)$,
\begin{equation*}
  (Tf)(k) = w(k)f(\varphi(k)), f \in C(K), k \in K,
\end{equation*}
 and let $\lambda \in \sigma(T) \setminus \{0\}$. The following conditions are equivalent.

(I) The operator $\lambda I - T$ is semi-Fredholm and $0 < def(\lambda I - T) < \infty$.

(II) There are subsets $L$ and $S$ of $K$ with the properties

\begin{enumerate}
\item The subsets $L$ and $S$ are at most finite and at least one of them is not empty.
  \item Every point of $L \cup S$ is an isolated point in $K$.
\item If $p, q \in L \cup S$ and $p \neq q$ then $\varphi^i(p) \neq \varphi^j(q), i, j \in \mathds{Z}$
  \item If $l \in L$ then the point $l$ not $\varphi$-periodic and $\lambda \in \sigma_r(T, C(tr(l)))$, where $Tr(l) = cl \{\varphi^j(l) ; j \in \mathds{Z}$.
  \item If $s \in S$ then the point $s$ is  $\varphi$-periodic and $\lambda^p = w_p(s)$ where $p$ is the period of the point $s$.
    \item Let $U = \bigcup \limits_{j= - \infty}^\infty \varphi^{(j)}(L \cup S)$ and $V = K \setminus U$. Then either $\lambda \not \in \sigma(T, C(V))$ or   $\lambda \in \sigma_r(T' C(V)')$.
\end{enumerate}
\end{theorem}

 \begin{remark} \label{r15}. Assume conditions of Theorem~\ref{t36}. Then $def(\lambda I - T) = card(L) + card(S)$.
 \end{remark}

 The case when $\lambda =0$, in other words the question when the operator $T$ is semi-Fredholm is resolved by the following proposition
 \begin{proposition} \label{p12} Let $T$ be a weighted automorphism of $C(K)$.
  \begin{equation*}
    (Tf)(k) = w(k)f(\varphi(k)), f \in C(K), k \in K.
  \end{equation*}
  The following conditions are equivalent.
\begin{enumerate}
  \item $T$ is semi-Fredholm.
  \item $T$ is Fredholm.
  \item $T$ is Fredholm and $ind\,  T =0$.
  \item The set $Z(w) = \{k \in K : w(k) = 0\}$ is at most finite and every point in $Z(w) is$ isolated in $K$.
\end{enumerate}
\end{proposition}

From the previous results the reader can easily derive a complete description of the essential spectra $\sigma_i(T), i = 1, \ldots , 5$ of a weighted automorphism $T$ of $C(K)$. The details can also be found in~\cite[Section 3, p.11]{Ki3}.

In Section 5 of the current paper it will be important for us that the above description of essential spectra of weighted automorphisms of $C(K)$ can be applied to a large class of disjointness preserving operators on Dedekind complete Banach lattices.

\begin{theorem} \label{t37} (see~\cite[Theorem 4.5, p.20]{Ki3}) Let $X$ be a Dedekind complete Banach lattice and $U$ be a $d$-automorphism of $X$ such that $\sigma(U) \subseteq \mathds{T}$. Let $w \in Z(X) \thickapprox C(K)$ and let $T = wU$. Consider the following automorphism of $C(K)$
\begin{equation*}
  f \rightarrow UfU^{-1}, f \in C(K),
\end{equation*}
  and let $\varphi$ be the corresponding homeomorphism of $C(K)$. Let $S$ be the weighted automorphism of $C(K)$ defined as
  \begin{equation*}
    (Tf)(k) = w(k)f(\varphi(k)), f \in C(K), w \in K.
  \end{equation*}
  Then
  \begin{equation*}
    \sigma(T) = \sigma(S) \; \text{and} \; \sigma_i(T) = \sigma_i(S), i = 1, \ldots , 5.
  \end{equation*}
\end{theorem}

\section{ Essential spectra of $d$-endomorphisms.}

Based on Theorem~\ref{t20} we make the following conjecture.

   \begin{conjecture} \label{con1}
    Let $X$ be a Banach $C(K)$-module and $T$ be a $d$-endomorphism of $X$. Assume that
    \begin{enumerate}
      \item The powers of $T$ are $d$-independent.
      \item $X$ has the Fatou property.
    \end{enumerate}
    Then the essential spectra $\sigma_i(T), i = 1.\, \ldots, 5$ are rotation invariant.
   \end{conjecture}

   So far we were unable to prove Conjecture~\ref{con1} in full generality but partial results in this direction are presented in the current section.

   We will need the following lemma.

   \begin{lemma} \label{l2} (see~\cite[Lemmas 12.6.1 - 12.6.3, p.92]{AAK}) Let $K$ be an extremally disconnected compact space, $E$ be a clopen subset of $K$, and $\varphi : E \rightarrow K$ be an open map. Let $\alpha \in \mathds{T}$.
   \begin{enumerate}
     \item If for some $m \in \mathds{N}$ the set $F^{(m)}$ is empty then there is an $f \in C(K)$ such that $|f| \equiv 1$ and
         \begin{equation*}
           |f(\varphi(k)) - \alpha f(k)| \leq \frac{2\pi}{m+1}, k \in E.
         \end{equation*}
     \item If $K$ is hyperstonean then for every $m \in \mathds{N}$ the set $F^{(m)}$ is clopen in $K$. Moreover, if we denote by $V_m$ the set $\varphi^{(-m)}(F^{(m)})$ then there is an $f \in C(K)$ such that $|f| \equiv 1$ on $K \setminus V_m$ and
         \begin{equation*}
           |f(\varphi(k)) - \alpha f(k)| \leq \frac{2\pi}{m+1}, k \in E.
         \end{equation*}
   \end{enumerate}

   \end{lemma}

     \begin{theorem} \label{t9} Let $X$ be a Banach $C(K)$-module and $T$ be a

      \noindent $d$-endomorphism of $X$. Assume that
  \begin{enumerate}
    \item The Banach $C(K)$-module $X$ has the Fatou property.
       \item The powers of $T$ are $d$-independent.
  \end{enumerate}

  Then the upper semi-Fredholm spectrum, $\sigma_2(T)$ is rotation invariant.
  \end{theorem}

  \begin{proof} Let $\lambda \in \sigma_2(T)$. We can assume without loss of generality that $\lambda =1$. We need to prove that $\mathds{T} \subseteq \sigma_2(T)$. Recall that the second conjugate space $X^{\prime \prime}$ is a Kaplansky $C(Q)$-module where $Q$ is a hyperstonean compact space and $T^{\prime \prime}$ is a $d$-endomorphism of $X^{\prime \prime}$. Moreover, there are a clopen subset $E$ of $Q$ and an open continuous map $\psi : E \rightarrow Q$ such that Equation~(\ref{eq54}) (see the discussion after Definition~\ref{d3}) holds.
    Let $Q_m = cl \bigcup \limits_{n=0}^\infty \psi^{(-n)}(F^{(m)})$ and $P_m = Q \setminus Q_m$. Notice that because $Q$ is hyperstonean and the map $\psi$ is open we have $\psi(Q_m) = \psi^{(-1)}(Q_m) = Q_m$. Let $p_m = \chi_{P_m}$ and $q_m = \chi_{Q_m}$. Then $p_m T^{\prime \prime} = T^{\prime \prime}p_m$. Let $T_m = p_mT^{\prime \prime}$. We consider two cases.

\noindent $(1)$ For any $m \in \mathds{N}$ we have $1 \in \sigma_2(T_m)$. Assume, contrary to our claim that there is an $\alpha \in \mathds{T}$ such that $\alpha \not \in \sigma_2(T)$. Because the set of all upper semi-Fredholm operators is open in $L(X)$ there is an $\varepsilon > 0$ such that if $A \in L(X)$ and $\|\alpha^{-1}T^{\prime \prime} - A\| < \varepsilon$ then $1 \not \in \sigma_2(A)$. Fix $m \in \mathds{N}$. By part 1 of Lemma~\ref{l2} there is an $f \in C(R_m)$ such that $|f| \equiv 1$ on $R_m$ and
\begin{equation}\label{eq55}
  |f(\psi(q)) - \alpha^{-1} f(q)| < \frac{2\pi}{m+1}, q \in E_m,
\end{equation}
where $E_m = R_m \cap E$.
Let
\begin{equation*}
  A = q_m\alpha^{-1}T^{\prime \prime} + f^{-1}T_m f.
\end{equation*}
Then applying~(\ref{eq55}) we see that
\begin{multline} \label{eq56}
  \|A - \alpha^{-1}T^{\prime \prime}\| =\|f^{-1}T_m f - \alpha^{-1}T_m \| = \|(\chi_{E_m}(f^{-1}(f \circ \psi) - \alpha^{-1})) T_m\|  \\
  \leq  \frac{2\pi}{m+1} \|T_m\| \leq  \frac{2\pi}{m+1} \|T\|.
  \end{multline}
If we choose $m$ so large that $ \frac{2\pi}{m+1} \|T\| < \varepsilon$ then it follows from~(\ref{eq56}) that $1 \not \in \sigma_2(A)$, and therefore, $1 \not \in \sigma_2(T_m)$, a contradiction.

$(2)$ There is an $m \in \mathds{N}$ such that $1 \not \in \sigma_2(T_m)$. Let $z_n \in X$ be a singular sequence such that $\|z_n\| = 1$ and $Tz_n - z_n \rightarrow 0$. Let $j$ be the canonical embedding of $X$ into $X^{\prime \prime}$. We can assume by choosing an appropriate subsequence that the sequence $p_mjz_n$ converges by norm in $X^{\prime \prime}$. Notice also that the sequence $t_n = (I - p_m)jz_n/\|(I - p_m)jz_n\|$ is singular and that $T^{\prime \prime} t_n - t_n \rightarrow 0$. For any $k \in \mathds{N}$ let $q_{k,m}$ be the projection corresponding to the characteristic function of the set $\psi^{(-k)}(F^{(m)})$. We have to consider two subcases.

\noindent $(2a)$ There are a positive constant $c$ and increasing sequences $n_l$ and $r_l$ of positive integers such that
\begin{equation}\label{eq26}
  \|(I - q_{n_l,m})t_{r_l}\| \geq c.
\end{equation}
We can assume without loss of generality that the integers $n_l$ are even. Indeed, if $n_l$ is odd we can change it to $n_l -1$ using the inequality
\begin{equation*}
   \|(I  - q_{n_l-1,m})t_{r_l}\| \geq  \|(I - q_{n_l,m})t_{r_l}\|.
\end{equation*}
Let $d_l = n_l/2$. Fix $\alpha \in \mathds{T}$ and consider function $f_l$ on $G_m$ defined as
\begin{equation}\label{eq27}
f_l(q) =  \left\{
    \begin{array}{ll}
      0, & \hbox{$q \in  \psi^{(-d_l)}(F^{(m)})$;} \\
      \alpha^i \big{(}1 - \frac{1}{\sqrt{d_l}} \big{)}^i , & \hbox{$q \in \psi^{(-n_l + i)}(F^{(m)}) \setminus \psi^{(-n_l +i + 1)}(F^{(m)}), i = 0, \ldots , n_l - d_l -1$;} \\
      \alpha^{-j}, & \hbox{$q \in \psi^{(-n_l -j-1)}(F^{(m)}) \setminus \psi^{(-n_l -j)}(F^{(m)}), j \in \mathds{N} $.}
    \end{array}
  \right.
\end{equation}
For any large enough $l$ we have
\begin{equation*}
  \big{(}1 - \frac{1}{\sqrt{d_l}}\big{)}^{d_l - 1} < \frac{1}{\sqrt{d_l}}
\end{equation*}
Then it follows from~(\ref{eq27}) by the means of a routine calculation that
\begin{equation}\label{eq57}
  |(f_l \circ \psi)(q) - \alpha f_l(q)| \leq \frac{1}{\sqrt{d_l}}, \; q \in \bigcup \limits_{n=1}^\infty \psi^{(-n)}(F^{(m)}).
\end{equation}
Because the compact space $Q_m$ is extremally disconnected there is a continuous extension $g_l$ of $f_l$ on $Q_m$, and it follows from~(\ref{eq57}) that
\begin{equation}\label{eq58}
  \|g_l \circ \psi - \alpha g_l\| \leq \frac{1}{\sqrt{d_l}}
\end{equation}
Let $y_l = g_l t_{r_l}$. We conclude from(~\ref{eq26}) and~(\ref{eq27}) that $\|y_l\| \geq c$. It follows from(~\ref{eq58}) that $T^{\prime \prime}y_l - \alpha y_l \rightarrow 0$, and finally, it follows from~(\ref{eq27}) that $supp(y_l) \cap \psi^{-d_l}(F^{(m)}) = \emptyset$, and therefore the sequence $y_l$ is singular.

\noindent $(2b)$ If we cannot find a constant $c$ and sequences $n_l$ and $r_l$ such that~(\ref{eq26}) holds, then we can assume without loss of generality that
\begin{equation}\label{eq28}
   \|(I - q_{n,m})t_n\| \leq 1/n, n \in \mathds{N}.
\end{equation}
Let $r, s \in \mathds{N}$ and let $z_r \neq z_s$. We claim that
\begin{equation}\label{eq59}
  \|z_r - z_s\| \leq \|jz_r -jz_s - q_{r,m}jz_r +q_{s,m}jz_s\|.
\end{equation}
To prove~(\ref{eq59}) let $u = z_r - z_s$ and let us consider the cyclic subspace $X(u)$ represented as a Banach lattice. Because the powers of $T$ are $d$-independent, according to Definition~\ref{d6} the set $\eta_u(\psi^{-n}(F^{(m)}\cup Q_u))$ is nowhere dense in $K_u$ for any $n \in \mathds{N}$. Therefore, we can find a net $y_\gamma$ of elements of the Banach lattice $X(u)$ such that $y_\gamma \uparrow u$ and for any $\gamma$ $q_{r,m}(j y_\gamma) = q_{s,m}(j y_\gamma) = 0$. Then
\begin{equation}\label{eq60}
  \|y_\gamma\| \leq \|jz_r -jz_s - q_{r,m}jz_r +q_{s,m}jz_s\|.
\end{equation}
But, because $X$ has the Fatou property, we have $\|y_\gamma\| \uparrow \|y\|$ and~(\ref{eq59}) follows from~(\ref{eq60}). From~(\ref{eq28}) and~(\ref{eq59}) we conclude that
\begin{equation}\label{eq29}
  \|z_r - z_s\| \leq \|p_m(jz_r - jz_s)\| +1/r + 1/s.
\end{equation}
Because $p_m(jz_n)$ is a Cauchy sequence it follows from~(\ref{eq29}) that the sequence $z_n$ is also a Cauchy sequence, in contradiction to our assumption that the sequence $z_n$ is singular.
  \end{proof}

  \begin{corollary} \label{c10} Let $X$ be a Kaplansky module with Fatou property, $T$ be a $d$-endomorphism of $X$, and $\varphi : E \rightarrow K$ be the corresponding map from Proposition~\ref{p3}. Assume one of the following conditions
  \begin{enumerate}
    \item The set of all eventually $\varphi$-periodic points is of first category in $K$.
    \item The operator $T$ is order continuous and the set of all $\varphi$-periodic points is of first category in $K$.
  \end{enumerate}
   Then the set $\sigma_2(T)$ is rotation invariant.
  \end{corollary}

  The next theorem complements the statements of Theorems~\ref{t22} and~\ref{t9}. We use the notations from Definition~\ref{d8}.

  \begin{theorem} \label{t11} Let $X$ be a Kaplansky module with Fatou property and $T$ be an order continuous $d$-endomorphism of $X$. Then
\begin{enumerate}
  \item $\gamma \sigma_i(T_m) = \sigma_i(T_m)$ and $\gamma \sigma_i(T_{(m)}) = \sigma_i(T_{(m)})$ for $i = 1, \ldots , 5$ and for every $\gamma \in \mathds{C}$ such that $\gamma^m = 1$.
  \item For every $m \in \mathds{N}$ we have $\sigma_2(T_{(m)}) \subseteq \sigma_2(T)$ and the set
$\sigma_2(T) \setminus \bigcup \limits_{m=1}^\infty \sigma_2(T_{(m)})$ is rotation invariant.
  \item For every $m \in \mathds{N}$ the set $\sigma_{a.p.}(T_{(m)}) \setminus \sigma_{a.p.}(T_m)$ is rotation invariant and contained in $\sigma_2(T)$. Moreover, if the set $\Pi_m$ contains no isolated points the same is true for the set $\sigma_{a.p.}(T_{(m)}) \setminus \sigma_2(T_m)$.
\end{enumerate}
\end{theorem}

\begin{proof}
  $(1)$  It is immediate to see that there is a clopen subset $\Pi$ of $\Pi_m$ such that the sets
  $\varphi^j(\Pi), j= 0, \ldots , m-1$ are pairwise disjoint and $\Pi_m = \bigcup \limits_{j=0}^{m-1} \varphi^j(\Pi)$. Let $\lambda \in \mathds{C}, \lambda^m = 1$. Consider the function $g$ defined on
$\bigcup \limits_{n=1}^\infty \varphi^{(-n)}(\Pi_m)$ as
  \begin{equation*}
  g(k)=  \left\{
      \begin{array}{ll}
        \lambda^j, & \hbox{if $k \in \varphi^j(|Pi)$}; \\
        \lambda^{-n}g(\varphi^n(k)), & \hbox{if $k \in \varphi^{-n}(\Pi_m) \setminus \varphi^{-n+1}(\Pi_m), n \in \mathds{N}$.}
      \end{array}
    \right.
  \end{equation*}
Let $h$ be the continuous extension of $g$ on $H_m$ and let
\begin{equation*}
r(k) =  \left\{
    \begin{array}{ll}
      h(k), & \hbox{ if $k \in H_m$;} \\
      1, & \hbox{otherwise.}
    \end{array}
  \right.
\end{equation*}
Then, $(r \circ \varphi)(k) = \lambda r(k), k \in H_m$. Therefore, $r^{-1}T_m r = \lambda T_m$ and
$r^{-1}T_{(m)} r = \lambda T_{(m)}$.

\noindent $(2)$ The inclusion $\sigma_2(T_{(m)} \subseteq \sigma_2(T)$ is trivial because the projection $q_m = \chi_{H_m}$ commutes with $T$. Let $\lambda \in \sigma_2(T) \setminus \bigcup \limits_{m=1}^\infty \sigma_2(T_{(m)})$. We can assume without loss of generality that $\lambda = 1$.
Let $L = cl \bigcup \limits_{m=1}^\infty H_m$ and $M =  K \setminus L$. Consider two possibilities.

\noindent $(2a)$ $1 \in \sigma_2(T, \chi_M X)$. Because $\varphi$ periodic points make a set of first category in $M$, by Corollary~\ref{c10} $\mathds{T} \subseteq \sigma_2(\chi_M X) \subseteq \sigma_2(T)$.

\noindent $(2b)$. $1 \not \in \sigma_2(T, \chi_M X)$. Then $1 \in \sigma_2(T, \chi_L X)$. Moreover, for any $m \in \mathds{N}$ we have $1 \in \sigma_2(T, \chi_{L_m}X$ where $L_m = cl \bigcup \limits_{j=m}^\infty H_j$. Therefore, we can find a sequence $x_m \in X$ such that $\|x_m\| = 1$,
$Tx_m - x_m \rightarrow 0$, and $supp \; x_m \subseteq L_m$. Fix $\alpha \in \mathds{T}$. By part (1) of Lemma~\ref{l2} for any $m \in \mathds{N}$ there is an $f_m \in C(L_m)$ such that $|f| \equiv 1$ and
$\|f \circ \varphi - \alpha f\|_{C(L_m)} \leq \frac{2\pi}{m}$. Then $\|f_m x_m\| = 1$ and it is easy to see that $T(f_m x_m) - \alpha (f_m x_m) \rightarrow 0$. Moreover, the sequence $f_m x_m$ is singular because $supp(f_mx_m) \subseteq L_m$.

\noindent $(3a)$. Let $\lambda \in \sigma_{a.p.}(T_{(m)}) \setminus \sigma_{a.p.}(T_m)$ and let $\alpha \in \mathds{T}$. Let $x_n \in X$, $\|x_n\| = 1$, $supp \; x_n \subseteq H_m$, and $Tx_n - \lambda x_n \rightarrow 0$. Because $\lambda \not \in \sigma_{a.p.}(T_m)$ we have
\begin{equation}\label{eq61}
  \chi_{\Pi_m}x_n \rightarrow 0.
\end{equation}
The proof of Part 3 of Proposition 12.15 in~\cite[p.98]{AAK} shows that~(\ref{eq61}) implies the existence of a sequence $y_n \in X$ such that $\|y_n\| = 1$, $Ty_n - \lambda \alpha y_n \rightarrow 0$ and $supp(y_n) \subseteq H_m \setminus \varphi^{(-n)}(\Pi_m)$. The last inclusion implies that the sequence $y_n$ is singular.

\noindent $(3b)$ Assume now that $\Pi_m$ has no isolated points and $\lambda \in \sigma_{a.p.}(T_{(m)}) \setminus \sigma_2(T_m)$. Let again $x_n \in X$, $\|x_n\| = 1$, $supp \; x_n \subseteq H_m$, and $Tx_n - \lambda x_n \rightarrow 0$. It is enough to prove that~(\ref{eq61}) holds. If not, then there is a $z \neq 0$, $supp \; z \subseteq \Pi_m$, such that $T_m z = \lambda z$. Because $\Pi_m$ contains no isolated points we can find an infinite sequence $E_j$ of pairwise disjoint clopen subsets of $supp \; z$ such that $\varphi(E_j) = E_j$. Then $T_m \chi_{E_j}g = \lambda \chi_{E_j}g$ and therefore $null(\lambda I - T_m) = \infty$, in contradiction with our assumption that $\lambda \not \in \sigma_2(T_m)$.
\end{proof}

\begin{remark} \label{r17}
  In the special case when the Kaplansky module $X$ is a Dedekind complete Banach lattice we have in virtue of Theorem~\ref{t23} that $\sigma(T_m) \subseteq \sigma(T), m \in \mathds{N}$. Moreover, because $T_m^m$ is a central operator, it is immediate to see that in case when $X$ has no atoms we have $\sigma(T_m) = \sigma_1(T_m)$. But, as the next example shows, the inclusion
$\sigma(T_m) \subseteq \sigma_2(T)$, or even $\sigma(T_m) \subseteq \sigma_{a.p.}(T)$, in general does not hold.
\end{remark}

\begin{example} \label{e3} Let $X = \ell^\infty(\mathds{N})$. We define the map $\varphi : \mathds{N} \rightarrow \mathds{N}$ as follows
\begin{equation*}
\varphi(k) =  \left\{
    \begin{array}{ll}
      1, & \hbox{if $k=1$;} \\
      k-1, & \hbox{if $k \in \mathds{N}$ and $k > 1$.}
     \end{array}
  \right.
\end{equation*}
The map $\varphi$ extends in the unique way to the open continuous surjection $\psi : K \rightarrow K$, where $K$ is the Gelfand compact of the Banach algebra $\ell^\infty(\mathds{N})$. Next, we define the weight $w \in \ell^\infty(\mathds{N})$ as
\begin{equation*}
w(k) =  \left\{
    \begin{array}{ll}
      1, & \hbox{if $k = 1$;} \\
      2, & \hbox{if $k \in \mathds{N}$ $k > 1$.}
     \end{array}
  \right.
\end{equation*}
 We denote by $W$ the element of $C(K)$ corresponding to $w$ and define the operator $T$ on $C(K)$ as
\begin{equation*}
  (Tf)(k) = W(k)f(\psi(k)), \; f \in C(K), \; k \in K.
\end{equation*}
Notice that $H_1 = \{1\}$ and $\sigma(T_1) = 1$. We claim that $1 \in \sigma_r(T)$. Indeed, assume to the contrary that $1 \in \sigma_{a.p.}(T)$ and let $x_n \in C(K)$ be such that $\|x_n\| = 1$ and $Tx_n - x_n \rightarrow 0$. Let $k_n \in K$ be a point such that $|x_n(k_n)| = 1$ and let $k$ be a limit point of the set $\{k_n\}$ in $K$. It is immediate to see that $|W_n(k)| \geq 1, n \in \mathbb{N}$ and also that for any $n \in \mathds{N}$ and for any $s \in \varphi^{(-n)}(\{k\})$ we have $|W_n(s)| \leq 1$. But it is clear from the definitions of $\varphi$ and $W$ that such a point $k$ does not exist. Actually, $\sigma(T) = 2\mathds{D}$ while $\sigma_r(T) = 2 \mathds{U}$.
\end{example}

\begin{remark} \label{r18}
  It is not difficult to modify Example~\ref{e3} in such a way that $X$ becomes a Dedekind complete Banach lattice ( and even a Banach function space) with no atoms. E.g. we can take as $X$ the Banach lattice $\ell^\infty (L^\infty (0,1))$.
\end{remark}

We can get stronger results by imposing additional conditions on the map $\varphi$ and/or the compact space $K$. The first such result is the following

    \begin{theorem} \label{t2} Let $X$ be a Kaplansky module, $T \in L(X)$ be a $d$-endomorphism of $X$, and $\varphi$ be the map from Proposition~\ref{p3}

Assume that the map $\varphi$ is open and the set of all periodic points of $\varphi$ is empty. Then the sets $\sigma(T), \sigma_{a.p.}(T), \sigma_r(T)$ as well as the essential spectra $\sigma_i(T), i = 1, \ldots , 5$ and $\sigma_2(T')$ are rotation invariant.
\end{theorem}

\begin{proof} Let $\alpha \in \mathds{T}$ and $\varepsilon >0$. By part (1) of Lemma~\ref{l2} there is a $g \in C(K)$ such that
\begin{equation}\label{eq8}
 |g| \equiv 1 \; \text{on} \; K \; \text{and} \; \|g\circ \varphi - \alpha g \| \leq \varepsilon.
\end{equation}
Then
\begin{equation}\label{eq9}
  \|\alpha T - g^{-1}Tg\| = \|\alpha T - g^{-1}(g\circ \varphi)T\|
\end{equation}
 Combining~(\ref{eq8}) and~(\ref{eq9}) we get
 \begin{equation}\label{eq10}
   \|\alpha T - g^{-1}Tg\| \leq \varepsilon \|T\|.
 \end{equation}
 The statement of the theorem follows from~(\ref{eq10}) and the well known fact that the set of all semi-Fredholm operators is open in $L(X)$ and the index of a semi-Fredholm operator is invariant under perturbations of small norm (see e.g.~\cite{Kat}).
\end{proof}

\begin{corollary} \label{c11} Let $X$ be a Kaplansky module, $T$ be a $d$-endomorphism of $X$, and
$\varphi$ be the map from Proposition~\ref{p3}. Assume that
\begin{enumerate}
  \item The compact space $K$ is hyperstonean.
  \item The powers of $T$ are $d$-independent.
\item The operator $T$ is order continuous.
\end{enumerate}
Then the statement of Theorem~\ref{t2} holds.
\end{corollary}

\begin{proof}
 Conditions (1) and (3) and Corollary A9 from~\cite[p.144]{AAK} imply that  the set $F^{(m)}$ is clopen in $K$ for any $m \in\mathds{N}$. Therefore, condition (2) implies that the sets $F^{(m)}$ are empty. Thus, the statement follows from Theorem~\ref{t2}.
\end{proof}

\begin{corollary} \label{c12} Let $X$ be a Kaplansky module and $T$ be a $d$-endomorphism of $X$. Assume that
\begin{enumerate}
  \item The compact space $K$ is hyperstonean.
  \item The operator $T$ is order continuous.
\end{enumerate}
Then the statement of Theorem~\ref{t11} holds.
\end{corollary}

\begin{proof}
  In the proof of Theorem~\ref{t11} we used the fact that $X$ has the Fatou property only in part (2a). But conditions (1) and (2) imply that the map $\varphi$ has no periodic points in $M$, where the set $M$ is defined as in the proof of Theorem~\ref{t11}, and we can apply Corollary~\ref{c11} instead of Corollary~\ref{c10}.
\end{proof}

An important special case of Corollary~\ref{c12} is the following

 \begin{corollary} \label{c1}. Let $X$ be a Banach $C(K)$-module and $T$ be a $d$-endomorphism of $X$. Assume that

\noindent $(*)$ Every cyclic subspace of $X$ when represented as a Banach lattice has order continuous norm.

\noindent Then the statement of Theorem~\ref{t11} remains correct.

   Moreover, if the powers of $T$ are $d$-independent then the sets $\sigma(T)$,

\noindent $\sigma_{a.p.}(T), \sigma_r(T)$, as well as the essential spectra $\sigma_i(T), i= 1, \ldots , 5$ and $\sigma_2(T^\prime)$ are rotation invariant.
 \end{corollary}

 \begin{proof} It follows from condition $(*)$ and Theorem 4.6 in~\cite{KO} that $X$ is a Kaplansky module and that the corresponding compact space $K$ is Hyperstonean. Moreover, condition $(*)$ implies that any $d$-endomorphism of $X$ is order continuous. It remains to apply Corollary~\ref{c12}.
 \end{proof}

\section{The case when the conjugate $T'$ also preserves disjointness.}

\noindent We start with the following simple result

 \begin{proposition} \label{p13} Let $X$ be a Kaplansky module and $T$ be an order continuous  $d$-endomorphism of $X$. Assume that

\noindent $(1)$ The map $\varphi : E \rightarrow K$ is injective.

\noindent $(2)$ The powers $T^i, i \in \mathds{N} \cup \{0\}$ are independent.

\noindent Then the sets $\sigma(T), \sigma_{a.p.}(T), \sigma_r(T)$, as well as the essential spectra $\sigma_i(T), i= 1, \ldots , 5$ and $\sigma_2(T^\prime)$ are rotation invariant.
 \end{proposition}

 \begin{proof} Condition $(1)$ and the Frolik's theorem (see~\cite[Theorem 6.25]{Wa}) imply that the set $F^{(m)}$ of all $\varphi$-periodic points of period less or equal to $m$ is clopen in $K$. Therefore, it follows from condition $(2)$ that the set of all $\varphi$-periodic points is empty.
It remains to apply Theorem~\ref{t2}.
 \end{proof}

\begin{theorem} \label{t39} Let $X$ be a Kaplansky module and $T$ be a $d$-endomorphism of $X$. Assume that
\begin{enumerate}
  \item The operator $T$ is order continuous.
  \item The conjugate operator $T^\prime$ is a $d$-endomorphism of the Kaplansky module $X^\prime$.
  \item The powers $T^i, i \in \mathds{N} \cup \{0\}$ are independent.
\end{enumerate}
Then the sets $\sigma(T), \sigma_{a.p.}(T), \sigma_r(T)$, as well as the essential spectra $\sigma_i(T), i= 1, \ldots , 5$ and $\sigma_2(T^\prime)$ are rotation invariant.
\end{theorem}

\begin{proof} For an arbitrary $x \in X$ consider the operator $T|X(x) \rightarrow X(Tx)$. Conditions (1) and (2) imply (see~\cite[Proof of Proposition 12.20, pp 104-105]{AAK}) that the conjugate operator
$(T|X(x))' : X(Tx)' \rightarrow X(x)'$ preserves disjointness. By a result of Arendt and Hart (see~\cite[Theorem 2.4]{AH}) the restriction of the map $\varphi$ on the set $supp(Tx)$ is injective. The set $\bigcup \limits_{x \in X} supp(Tx)$ is dense in $E$. Because $E$ is extremally disconnected, $\varphi$ is a homeomorphism of $E$ onto $\varphi(E)$. It remains to apply Proposition~\ref{p13}.
\end{proof}

\begin{remark} \label{r25} Theorem 2.4 in~\cite{AH} is stated for the case when the domain and the range is the same Dedekind complete Banach lattice. But their proof (based on Theorem 3.1 in~\cite{LS1}) remains valid when the domain and the range are distinct Dedekind complete Banach lattices.
\end{remark}

In the case when $T$ is a $d$-endomorphism of a Kaplansky module $X$  and the map $\varphi$ is injective, the cited above Frolik's theorem implies that $\Pi_m = P_m$ \footnote{We use the notations from Definitions~\ref{d9} and~\ref{d8}.}. Obviously, in this case $H_m = \Pi_m$ and therefore, $T_m = T_{(m)}$.

\begin{theorem} \label{t10} Let $X$ be a Kaplansky module and $T$ be an order continuous $d$-endomorphism of $X$. Assume that the map $\varphi : E \rightarrow K$ is injective. Assume additionally one of the following conditions
\begin{enumerate} [(a)]
  \item The module $X$ has the Fatou property.
  \item The compact space $K$ is hyperstonean.
\end{enumerate}
 Then
\begin{enumerate}
  \item $ \sigma(T) = \bigcup \limits_{m=1}^\infty \sigma(T_m) \cup \sigma_\infty$,
  \item $ \sigma_{a.p.}(T) = \bigcup \limits_{m=1}^\infty \sigma_{a.p.}(T_m) \cup \sigma_{a.p., \infty}$,
  \item $ \sigma_i(T) = \bigcup \limits_{m=1}^\infty \sigma_i(T_m) \cup \sigma_{i, \infty}, i=1, \ldots , 5$,
  \item $ \sigma_2(T') = \bigcup \limits_{m=1}^\infty \sigma_2(T_m^\prime) \cup \sigma_{2, \infty}^\prime$.
\end{enumerate}
where the sets $\sigma_\infty, \sigma_{ap, \infty}, \sigma_{i, \infty}, i =1, \ldots, 5$, and $\sigma_{2, \infty}^\prime$ are rotation invariant.
\end{theorem}

\begin{proof} The inclusions
\begin{multline*}
  \sigma(T_m) \subseteq \sigma(T), \sigma_{a.p.}(T_m) \subseteq \sigma_{a.p.}(T), \\
\sigma_i(T_m) \subseteq \sigma_i(T), i = 1, \ldots , 5, \; \text{and} \; \sigma_2(T_m^\prime) \subseteq \sigma_2(T')
\end{multline*}
follow from the fact that the projections $\chi_{\Pi_m}$ commute with $T$. To prove statements (1) - (5) let us introduce the sets $R_n = K \setminus \bigcup \limits_{j=1}^n \Pi_j$ and notice that in virtue of Lemma~\ref{l2} for any given positive $\varepsilon$ and any $\alpha \in \mathds{T}$ we can find an $n \in \mathds{N}$ and an $f \in C(R_n)$ such that $|f| \equiv 1$ on $R_n$ and
$\|f \circ \varphi - \alpha f\|_{C(R_n)} \| < \varepsilon$. After that we can proceed as in the proof of Theorem~\ref{t2}.
 \end{proof}

\begin{corollary} \label{c4} Let $X$ be a Kaplansky module and $T$ be a $d$-endomorphism of $X$. Assume that
\begin{enumerate}
  \item The operator $T$ is order continuous.
  \item The conjugate operator $T^\prime$ is a $d$-endomorphism of the Kaplansky module $X^\prime$.
\end{enumerate}
Then the conclusion of Theorem~\ref{t10} holds.
\end{corollary}

Let us add one more result in this direction. Notice that in the following theorem we do not assume that $T$ is order continuous and therefore Theorem 2.4 from~\cite{AH} cannot be directly applied.

\begin{theorem} \label{t12} Let $X$ be a Kaplansky module and $T$ be a $d$-endomorphism of $X$. Assume that
\begin{enumerate}
  \item $X$ has the Fatou property.
  \item $T^\prime$ is a $d$-endomorphism of the Kaplansky module $X'$.
\end{enumerate}
Then
\begin{equation*}
  \sigma_2(T) = \bigcup \limits_{m=1}^\infty \sigma_2(T_m) \cup \sigma_{2, \infty}
\end{equation*}
where the set $\sigma_{2, \infty}$ is rotation invariant.

\noindent In particular, if the powers of $T$ are $d$-independent the set $\sigma_2(T)$ is rotation invariant.
\end{theorem}

\begin{proof} It was proved in~\cite[Proof of Proposition 12.20, p.105]{AAK} that conditions (1) and (2) imply that $\varphi^{-1}(\Pi_m) = \Pi_m$. Therefore, $H_m = \Pi_m$ and we can proceed as in the proof of part (2) of Theorem~\ref{t11}. Indeed, in part $(2a)$ of that proof we did not use the fact that $T$ is order continuous, abd in part $(2b)$ we used it only to establish that $\varphi^{-1}(H_m) = H_m$.
\end{proof}

In view of Theorems~\ref{t10} and~\ref{t12} it becomes desirable to obtain more information about the spectrum and essential spectra of operators $T_m$.

\begin{proposition} \label{p14} Let $X$ be a Kaplansky module and assume that $T \in L(X)$ commutes with all the operators from $C(K)$.
\begin{enumerate}
  \item If $K$ has no isolated points then
\begin{equation} \label{eq62}
  \lambda \in \sigma_p(T) \Rightarrow null(\lambda I - T) = \infty,
\end{equation}
and
\begin{equation} \label{eq63}
  \lambda I - T \in \Phi_+ \Rightarrow null(\lambda I - T) =0.
\end{equation}
  \item If $Q$ does not have isolated points, where $C(Q) = Z(X')$, then in addition to~(\ref{eq62}) and~(\ref{eq63}) we have
\begin{equation}\label{eq64}
   \lambda \in \sigma_p(T') \Rightarrow def(\lambda I - T) = \infty,
\end{equation}
and
\begin{equation} \label{eq65}
  \lambda I - T \in \Phi_- \Rightarrow def(\lambda I - T) =0.
\end{equation}
\end{enumerate}
\end{proposition}

\begin{proof} $(1)$ Let $Tx = \lambda x$, where $x \in X$ and $x \neq 0$. Then $supp \; x$ is an infinite clopen subset of $K$. Let $E_j$ be the sequence of pairwise disjoint clopen, nonempty subsets of $supp \; x$. Then for every $n$ we have $x_n = \chi_{E_n} x \neq 0$ and $Tx_n = \lambda x_n$. Therefore, $null(\lambda I - T) = \infty$ and $\lambda I - T \not \in \Phi_+$.

$(2)$ Because $T$ commutes with $C(K)$ and $T'$ is continuous in the topology $\sigma(L(X'), X' \otimes X)$, $T'$ commutes with $C(Q)$. It is immediate to see that if $Q$ does not have isolated points, neither does $K$. It remains to apply the reasoning from part $(1)$ of the proof.
\end{proof}

\begin{corollary} \label{c15} Let $X$ be a Kaplansky module and let $T \in L(X)$ commute with
  $C(K)$. Assume that $Q$ has no isolated points. Then
\begin{enumerate}
  \item $\sigma_2(T) = \sigma_{a.p.}(T)$.
  \item $\sigma_2(T') = \sigma_{a.p.}(T')$.
  \item If $\lambda I - T$ is a Fredholm operator then it is invertible. Therefore, $\sigma_3(T) = \sigma(T)$.
\end{enumerate}
\end{corollary}

\begin{corollary} \label{c13} Let $X$ be a Kaplansky module and $T$ be a $d$-endomorphism of $X$.
If $K$ has no isolated $\varphi$-periodic points then for every $m \in \mathds{N}$ such that $\Pi_m \neq \emptyset$ we have $\sigma_2(T_m) = \sigma_{a.p.} (T_m)$.

If $Q$ does not have isolated $\psi$-periodic points then additionally we have $\sigma_2(T_m') = \sigma_{a.p.}(T_m')$ and $\sigma_3(T_m) = \sigma(T_m)$.
 \end{corollary}

\begin{proof} The proof follows from Corollary~\ref{c15} and the fact that the operator $T_m^m$ commutes with the operators from $C(K)$.
\end{proof}

\begin{corollary} \label{c14}
  Let $X$ be a Kaplansky module and $T$ be a $d$-isomorphism of $X$. Assume that $K$ has no isolated $\varphi$-periodic points. Assume also that every cyclic subspace, $X(x)$, represented as a Banach lattice, has order continuous norm. Then for every
$m \in \mathds{N}$ such that $\Pi_m \neq \emptyset$ we have $\sigma_2(T_m) = \sigma_{a.p.} (T_m)$,
 $\sigma_2(T_m') = \sigma_{a.p.}(T_m')$, and $\sigma_3(T_m) = \sigma(T_m)$.
\end{corollary}

\begin{proof} By Theorem 9.10 from~\cite[p.61]{AAK} the conditions of the corollary imply that
$Z(X') = Z(X)$. It remains to apply Corollary~\ref{c13}.
\end{proof}

We can prove a stronger result.

\begin{proposition} \label{p15} Let $X$ be a Kaplansky module and $T$ be a

\noindent $d$-endomorphism of $X$. Assume that for every $x \in X$ the conjugate space $X(x)'$, represented as a Banach lattice, has no atoms. Then then for every $m \in \mathds{N}$ such that $\Pi_m \neq \emptyset$ we have $\sigma_2(T_m) = \sigma_{a.p.} (T_m)$, $\sigma_2(T_m') = \sigma_{a.p.}(T_m')$, and $\sigma_3(T_m) = \sigma(T_m)$.
\end{proposition}

\begin{proof} Because  for any $x \in X$ $X(x)'$ has no atoms, neither does $X(x)$, and therefore $K$ has no isolated points. Assume that $q$ is an isolated point of $Q$ and consider $z \in X'$ such that
$supp \; z = \{q\}$. It is immediate to see that for any $g \in C(Q)$ we have $gz = g(q)z$. Let $x \in X$ be such that $z(x) =1$. Then the restriction $u = z|X(x)$ is a positive functional on the Banach lattice $X(x)$. Indeed, if $f \in C(K)$ and $f \geq 0$ then $u(fx) = f'z(x) = f'(q)z(x) \geq 0$. Assume that $v \in X(x)'$ and $0 \leq v \leq u$. Then we have $f \in C(K), f(q) = 0 \Rightarrow v(fx) = 0$. Let $a = v(x)$. We can easily verify that $v = au$ and thus $u$ is an atom in $X(x)'$ in contradiction with our assumptions.
\end{proof}

\begin{remark}  \label{r20} Necessary and sufficient conditions for the conjugate $X'$ to a Banach lattice $X$ to be non-atomic were obtained by Lozanovsky in~\cite{Lo} in the case when $X$ is a Dedekind complete Banach lattice, and in general case by de Pagter and Wnuk in~\cite{PW}.
\end{remark}

Let us turn to the case when $K$ does have isolated $\varphi$-periodic points. In this case, without any additional assumptions we cannot make any meaningful statements about the spectrum and essential spectra of $T$. Indeed, if $p$ is an isolated point in $K$ and $\varphi(p) = p$ then $Y = \chi_{\{p\}}X$ can be an arbitrary Banach space and $\chi_{\{p\}}T$ can be an arbitrary operator from $L(Y)$. But the situation changes if we assume that the projections corresponding to isolated $\varphi$-periodic points are finite dimensional.

The proof of the following Theorem~\ref{t50} follows immediately from Proposition~\ref{p15}, and therefore we omit it. The statement of Theorem~\ref{t50} requires one explanation. It is easy to see that for every $m \in \mathds{N}$ the operator $T_m^\prime$ is a $d$-endomorphism of $X'$ and $(T_m^\prime)^m$ commutes with $C(Q)$. We denote the corresponding map by $\psi$. The map $\psi$ is defined on $\chi_{\Pi_m}'Q$ and every point of $\chi_{\Pi_m}'Q$ is $\psi$-periodic with the period $m$.

\begin{theorem} \label{t50} Let $X$ be a Kaplansky module and $T$ be a $d$-endomorphism of $X$. Assume that for any isolated $\varphi$-periodic point of $K$ we have $\dim{\chi_{\{p\}}}X < \infty$.
Then for any $m \in \mathds{N}$ such that $\Pi_m \neq \emptyset$,
\begin{enumerate}
  \item \begin{equation*}
\sigma_5(T_m) = \sigma_3(T_m).
\end{equation*}
  \item The following conditions are equivalent
\begin{enumerate} [(a)]
\item
\begin{equation*}\label{eq66}
  \lambda \in \sigma(T_m) \setminus \sigma_5(T_m).
\end{equation*}
\item There are isolated points $p_1, \ldots , p_k \in \Pi_m$ such that
\begin{multline*}
\varphi^s(p_i) \neq \varphi^t(p_j),1\leq i < j \leq p, s,t \in [1 : m], \\
\lambda^m \in \sigma((\chi_{p_i}T_m)^m), i= 1, \ldots , p, \\
\lambda \not \in \sigma(\chi_ET_m), \; \text{where} \; E = \Pi_m \setminus \{\varphi^j(p_i), i \in [1 : k], j \in [1 : m]\}.
\end{multline*}
\item There are isolated points $q_1, \ldots , q_k \in \chi_{\Pi_m}Q$ such that
\begin{multline*}
\psi^s(q_i) \neq \psi^t(p_j),1\leq i < j \leq p, s,t \in [1 : m], \\
\lambda^m \in \sigma((\chi_{q_i}T_m^\prime)^m), i= 1, \ldots , p, \\
\lambda \not \in \sigma(\chi_ET_m^\prime), \; \text{where} \; E = \chi_{\Pi_m}Q \setminus \{\varphi^j(q_i), i \in [1 : k], j \in [1 : m]\}.
\end{multline*}
  \end{enumerate}
\end{enumerate}
\end{theorem}

There is an important subclass of the class of Kaplansky modules for which the condition of Theorem~\ref{t50} is automatically satisfied. Let us recall the following definition from~\cite{KO1}.

\begin{definition} \label{d10} A Banach $C(K)$-module $X$ is called \textbf{finitely generated} if there are $x_1, \ldots, x_n \in X$ such that the linear subspace $\sum \limits_{i=1}^n X(x_i)$ is dense in $X$.
\end{definition}

\begin{remark} The reader interested in the properties of finitely generated $C(K)$-modules is referred
to~\cite{KO} -~\cite{KO3}.
\end{remark}

\begin{corollary} \label{c16} Let $X$ be a finitely generated Kaplansky module and $T$ be a $d$-endomorphism of $X$. Then for every isolated point $p$ in $K$ we have $\dim{\chi_{\{p\}}}X < \infty$ and therefore the conclusion of Theorem~\ref{t50} remains valid.
\end{corollary}

We can get a complete description of the spectrum and essential spectra of operators $T_m$  in the case when $T_m^m \in C(K)$. It is the case, in particular, when $X$ is a Dedekind complete Banach lattice. Taking into consideration statement (1) of Theorem~\ref{t11}, we see that it is sufficient to prove the following proposition.

\begin{proposition} \label{p17} Let $X$ be a Banach $C(K)$-module and $f \in C(K)$. Then
\begin{enumerate}
  \item $\sigma(f) = f(K)$.
  \item $\sigma_1(f) = \sigma_5(f)$.
  \item The following conditions are equivalent
  \begin{enumerate}[(a)]
    \item $\lambda \in \sigma(f) \setminus \sigma_5(f)$.
    \item There is a finite nonempty subset $S= \{k_1, k_2, \ldots , k_n\}$ of $K$ such that
        \begin{itemize}
          \item every point of $S$ is isolated in $K$.
          \item $f(k_i) = \lambda, i = 1, \ldots , n$.
          \item $\dim{\chi_{\{k_i\}}X} < \infty, i = 1, \ldots , n$.
          \item $\lambda \not \in f(K \setminus S)$.
        \end{itemize}
  \end{enumerate}
\end{enumerate}
\end{proposition}

\begin{proof} (1) is obvious.

\noindent The implication $(3b) \Rightarrow (3a)$ is also obvious.

\noindent Assume that $(3b)$ is false and consider the set $K_\lambda = f^{-1}(\lambda)$. Then either $K_\lambda$ contains a point $k$ isolated in $K$ such that $\dim{\chi_{\{k\}}X} = \infty$, and therefore $\lambda \in \sigma_5(T)$, or it contains a limit point of $K$ and in this case $\lambda \in \sigma_5(T)$ follows from (1).

\noindent It remains to prove that if $\lambda \in \sigma_5(T)$ then $\lambda \in \sigma_1(T)$. From the previous step it is clear that we can consider the case when $K_\lambda$ contains a point $k$ which is not isolated in $K$. Let $O_n, n \in \mathds{N}$, be  open subsets of $K$ with properties
\begin{enumerate}[(i)]
  \item $cl O_k \cap cl(\bigcup \limits_{n \neq k} O_n) = \emptyset$.
  \item $|f(t) - \lambda| < 1/n, t \in O_n$.
  \end{enumerate}
Consider $f_n \in C(K)$ such that $\|f_n\|=1$ and $supp (f_n) \subseteq O_n$. Let $x_n \in X$ be such that $\|x_n\| =1$ and $\|f_n x_n \geq 1 - 1/n$. Let $y_n = f_nx_n$. Then $\|y_n\| \rightarrow 1$ and it follows from $(ii)$ that  $fy_n - \lambda y_n \rightarrow 0$.

Next we will prove that the sequence $y_n$ is singular. Let $x = \sum \limits_{n=1}^\infty y_n/n^2$ and let $X(x)$ be the corresponding cyclic subspace of $X$, i.e.
\begin{equation*}
  X(x) = cl \{gx : g \in C(K)\}.
\end{equation*}
The space $X(x)$ endowed with the original norm on $X$ and the order generated by the cone $X_+(x) = cl \{gx : g \in C_+(K)\}$ is a Banach lattice (see~\cite{HO}). It follows easily from $(i)$ that $y_n \in X(x)$ and $y_n$ are disjoint elements of $X(x)$. Therefore the sequence $y_n$ is singular and $\lambda \in \sigma_2(T)$.

To prove that $\lambda \in \sigma_2(T^\prime)$ consider $x_n^\prime \in X^\prime$ such that $\|x_n^\prime\| = 1$ and $\|f_n^\prime x_n^\prime\| \geq 1 - 1/n$. Let $y_n^\prime = f_n^\prime x_n^\prime$. Then $f^\prime y_n^\prime - \lambda y_n^\prime \rightarrow 0$. Let $x^\prime = \sum \limits_{n=1}^\infty y_n^\prime/n^2$. Then $y_n^\prime$ are disjoint elements of the Banach lattice $X^\prime(x^\prime)$ and therefore the sequence $y_n^\prime$ is singular.
\end{proof}

\section{ The case $T = wU$, $\sigma(U) \subseteq \mathds{T}$.}

The main goal of the results presented in this section is to obtain an analog of Theorem~\ref{t37} for Kaplansky modules.

Let $X$ be a Kaplansky module and $T$ be a $d$-endomorphism of $X$. The operator $T$ is called a
$\mathbf{d}$-\textbf{isomorphism} of $X$ if $T$ is invertible and the inverse $T^{-1}$ is also a $d$-endomorphism of $X$.

\begin{remark} If $X$ is a Banach lattice and $T$ is an invertible disjointness preserving operator on $X$ then, as proved by Huijsmans and de Pagter in~\cite{HP}, the inverse operator automatically preserves disjointness. It is in general not true for invertible $d$-endomorphisms of Banach $C(K)$-modules. See~\cite[Remark 12.17, p.99]{AAK} for a simple example of an invertible $d$-endomorphism of a Kaplansky module such that the inverse operator does not preserve disjointness.
\end{remark}

In this section we will consider a special class of $d$-isomorphisms. Namely, we will assume that $X$ is a Kaplansky module, $U$ is a $d$-isomorphism of $X$, and
\begin{equation}\label{eq35}
   \sigma(U) \subseteq \mathds{T}.
\end{equation}
The operators on $X$ considered in this section are of the form

\begin{equation}\label{eq30}
  T = wU, w \in C(K).
\end{equation}
The map
\begin{equation*}
  f \rightarrow UfU^{-1}
\end{equation*}
defines an automorphism of $C(K)$. We will denote the corresponding homeomorphism of $K$ by $\varphi$.

\noindent Together with $T$ we will consider operator $S$ on $C(K)$ defined as
\begin{equation}\label{eq31}
  (Sf)(k) = w(k)f(\varphi(k)), \; f \in C(K), \; k \in K.
\end{equation}
The operator $S$ is a weighted automorphism of $C(K)$. The spectrum and essential spectra of such operators are completely described by Theorems~\ref{t35}, ~\ref{t33}, \ref{t34}, and~\ref{t36}. See also~\cite{Ki1} and~\cite{Ki3}. We are interested in the relations between the spectra of $T$ and the corresponding spectra of $S$.
We cannot expect the statement of Theorem~\ref{t37} to hold in the case when $X$ is an arbitrary Kaplansky module, at least, without any additional assumptions. Indeed, let $H$ be a Hilbert space such that $\dim{H} \geq 2$ and $U$ be a unitary operator on $H$ such that $\sigma(U)$ is not a singleton. We consider $H$ as a $C(K)$-module where $K$ is a singleton. Let $w = 1$. Then clearly
$\sigma(S) = \{1\} \neq \sigma(U)$. Nevertheless, by putting some reasonable restrictions on $\varphi$ we will be able to obtain a considerable amount of information about the spectra of $T$ and their relation to the corresponding spectra of $S$.

\begin{theorem} \label{t15} Assume that the map $\varphi$ has no periodic points in $K$. Then
  \begin{equation*}
    \sigma(T) = \sigma(S).
  \end{equation*}
\end{theorem}

\begin{proof} It follows immediately from~(\ref{eq35}) that $\rho(T) = \rho(S)$. Because the set of $\varphi$-periodic points is empty, it follows from Theorem~\ref{t20} that both $\sigma(S)$ and $\sigma(T)$ are rotation invariant. Thus, our statement will be proved if we show that
\begin{equation}\label{eq33}
  |\sigma(T)| = |\sigma(S)|,
\end{equation}
where
\begin{equation*}
  |\sigma(A)| = \{|\lambda| : \; \lambda \in \sigma(A), \; A \in L(X)\}.
\end{equation*}
Assume first that both $T$ and $S$ are invertible operators, i.e. $|w| > 0$ on $K$. Notice that $\rho(T^{-1}) = \rho(S^{-1})$. Let $r >0$ and $r \not \in |\sigma(S)|$. We need to prove that $r \not \in |\sigma(T)|$. The cases when
\begin{equation*}
  r > \rho(S) = \rho(T) \; \text{or} \; r < 1/\rho(S^{-1}) = 1/\rho(T^{-1})
\end{equation*}
are trivial
 Thus we can assume that
  \begin{equation*}
    1/\rho(S^{-1}) < r < \rho(S).
  \end{equation*}
  By Theorem 3.10 in~\cite{Ki1} $K$ is the union of two clopen $\varphi$-invariant subsets $K_1$ and $K_2$ such that
\begin{multline}\label{eq34}
  \sigma(\chi_1S, C(K_1)) \subseteq \{\lambda \in \mathds{C}: |\lambda| < r\} \\
 \;\text{and} \; \sigma(\chi_2S, C(K_2)) \subseteq \{\lambda \in \mathds{C}: |\lambda| > r\},
\end{multline}
where $\chi_1$ and $\chi_2$ are the characteristic functions of $K_1$ and $K_2$, respectively. It is easy to see from~(\ref{eq35}) and~(\ref{eq34}) that
\begin{multline}\label{eq36}
  \sigma(\chi_1T, \chi_1X) \subseteq \{\lambda \in \mathds{C}: |\lambda| < r\} \\
 \;\text{and} \; \sigma(\chi_2T, \chi_2X) \subseteq \{\lambda \in \mathds{C}: |\lambda| > r\}.
\end{multline}
It is immediate to see from~(\ref{eq36}) and the fact that $T$ commutes with $\chi_1$ and $\chi_2$ that $r \not \in |\sigma(T)|$.

Now assume that $r \not \in |\sigma(T)|$ and $ 1/\rho(T^{-1}) < r < \rho(T)$. We claim that $r \not \in |\sigma(S)|$. We will prove this claim in several steps.

\noindent (1) Notice that $\sigma(T) = \sigma_1 \cup \sigma_2$ where $\sigma_1 \subset \{\lambda \in \mathds{C} : |\lambda| < r \}$ and $\sigma_2 \subset \{\lambda \in \mathds{C} : |\lambda| > r \}$. Let $P_1, P_2$ be the spectral projections corresponding to the sets $\sigma_1, \sigma_2$, respectively.  On this step we claim that $P_1$, and therefore, $P_2$ commute with operators from $C(K)$. Let $x \in P_1X$ and $f \in C(K)$. Then for any $n \in \mathds{N}$ we have
\begin{equation*}
  T^nfx = w_nU^nfx = w_nU^nfU^{-n}U^nx =w_n(f \circ \varphi^n)U^nx = (f \circ \varphi^n)T^nx,
\end{equation*}
 Therefore,
\begin{equation}\label{eq37}
  \limsup \|T^nfx\|^{1/n} \leq r.
\end{equation}
It follows from~(\ref{eq37}) that $P_1fP_1 = fP_1$. Therefore $P_1^\prime f^\prime P_1^\prime = P_1^\prime f^\prime$. On the other hand let $F \in P_1^\prime X^\prime$. Then for any $n \in \mathds{N}$
\begin{multline}\label{eq38}
  (T^\prime)^nf^\prime F = (U^n)^\prime (w_n)^\prime f^\prime F = (U^n)^\prime f^\prime (U^{-n})U^n (w_n)^\prime F = \\
  = (f \circ \varphi^{-n})^\prime (T^n)^\prime F.
\end{multline}
It follows from~(\ref{eq38}) that
 \begin{equation}\label{eq39}
  \limsup \|(T^n)^\prime f^\prime F\|^{1/n} \leq r,
\end{equation}
and therefore $P_1^\prime f^\prime P_1^  =  f^\prime P_1^\prime$. Hence, $P_1f = fP_1$.

\noindent (2) On this step we will prove that $P_1$ commutes with $U$, and therefore, with $U^{-1}$. Let $x \in P_1X$. It follows immediately from the identity
\begin{equation*}
  T^n U x =(w \circ \varphi^n)^{-1} T^{n+1} x
\end{equation*}
that $Ux \in P_1x$ and thus $P_1UP_1 = UP_1$. Similarly, if $F \in P_1^\prime X^\prime$ then
\begin{multline} \label{eq40}
  (T^n)^\prime U^\prime F = (U^n)^\prime (w_n)^\prime U^\prime F = (U^{n+1})^\prime (U^{-1})^\prime (w_n)^\prime U^\prime F = \\
  = (U^{n+1})^\prime (w_n \circ \varphi)^\prime F = (T^{n+1})^\prime (w^\prime)^{-1}F.
\end{multline}
Because, by step 1, $(w^\prime)^{-1}F \in P_1^\prime X^\prime$, it follows from~(\ref{eq40}) that $U^\prime F \in P_1 X^\prime$. We conclude that $P_1U = UP_1$ and therefore $P_1U^{-1} = U^{-1}P_1$.

\noindent (3) On this step we will assume, contrary to what we intend to prove, that $r \in |\sigma(S)|$. Then $r\mathds{T} \subseteq \sigma(S)$. There are two possibilities.

(3a) $r\mathds{T} \subseteq \sigma_{a.p.}(S)$. Then by Lemma 3.25 from~\cite{Ki1} there is a point $k$ in $K$ such that
\begin{equation}\label{eq41}
  |w_n(k)| \geq r^n, \; |w_n(\varphi^{-n}(k))| \leq r^n, \; n \in \mathds{N}.
\end{equation}
As $k$ is not a $\varphi$-periodic point for any $m \in \mathds{N}$ we can find a clopen neighborhood $O_m$ of $k$ such that the sets $\varphi^{(j)}(O_m), |j| \leq m+1$ are pairwise disjoint and
\begin{equation}\label{eq48}
 |w_n(s)| \geq (1/2)r^n, \; |w_n(\varphi^{-n}(s)| \leq 2r^n, \; s \in O_m, n \in [1 : m+1].
\end{equation}
 Let $x_m \in X$ be such that $\chi_{O_m}x_m=x_m$ and $\|x_m\| = 1$, and let
\begin{equation}\label{eq42}
  y_m = \sum \limits_{j=-m}^m  \big{(}1 - \frac{1}{\sqrt{m}}\big{)}^{|j|} r^{-j} T^j x_m.
\end{equation}
It follows from~(\ref{eq35}), (\ref{eq48}), and~(\ref{eq42}) that (see~\cite{Ki1} for more details)
\begin{equation*}
\|Ty_m - ry_m\| = o(\|y_m\|), m \rightarrow \infty
\end{equation*}
  Hence, $r \in |\sigma(T)|$, a contradiction.

(3b) $r\mathds{T} \subseteq \sigma_r(S)$. Then, by Theorem 3.29 in~\cite{Ki1}, $r \in \sigma_{a.p.}(\tilde{S})$ where
\begin{equation*}
  (\tilde{S}f)(k) = w(k)f(\varphi^{-1}(k)), f \in C(K), k \in K.
\end{equation*}
Consider on $X$ the operator $\tilde{T} = wU^{-1}$. By steps (1) and (2) $\tilde{T}$ commutes with projections $P_1$, $P_2$. Therefore,
\begin{equation}\label{eq43}
  \|\tilde{T}^n P_1\| = \|w_nU^{-n}P_1\| = \|T^nP_1U^{-2n}\| \leq \|T^nP_1\|\|U^{-2n}\|.
\end{equation}
From~(\ref{eq43}), (\ref{eq35}), and the fact that $\rho(P_1T) < r$ we obtain that
\begin{equation}\label{eq44}
  \rho(P_1\tilde{T}) < r.
\end{equation}
 Similarly we can prove that
 \begin{equation}\label{eq45}
   \rho(P_2\tilde{T}^{-1}) < r.
 \end{equation}
 From~(\ref{eq44}) and~(\ref{eq45}) follows that $r\mathds{T} \cap \sigma(\tilde{T}) = \emptyset$. It follows from the results in~\cite{Ki1} that $\sigma(\tilde{S}) = \sigma(S)$. Applying step (3a) to operators $\tilde{T}$ and $\tilde{S}$ we come to a contradiction.

 To finish the proof it remains to consider the case when the weight $w$ is not invertible in $C(K)$. No changes are needed to prove the implication
 \begin{equation*}
   r \not \in |\sigma(S)| \Rightarrow r \not \in |\sigma(T)|.
 \end{equation*}
 Assume that $0 < r < \rho(T)$ and that $r \not \in |\sigma(T)|$. Because the compact space $K$ is extremally disconnected we can for every $n \in \mathds{N}$ find $w^{(n)} \in C(K)$ such that $w^{(n)}$ is invertible in $C(K)$ and $\|w - w^{(n)}\| < 1/n$. Let $T^{(n)}= w^{(n)}U$. Then for any large enough $n$ we have $r \not \in |\sigma(T^{(n)})|$. Let $S^{(n)}$ be the operator on $C(K)$ defined as
 \begin{equation*}
   (S^{(n)}f)(k) = w^{(n)}(k)f(\varphi(k)), f \in C(K), k \in K.
 \end{equation*}
 By the previous part of the proof for any large enough $n$ we have
 \begin{equation}\label{eq46}
   r \not \in |\sigma(S^{(n)}|.
 \end{equation}
 Assume that $r \in |\sigma(S)|$. It follows from~(\ref{eq46}) and from the fact that the set of operators invertible from the left is open in $L(X)$ that $r \in \sigma_{a.p.}(S)$. Let $k \in K$ and $O_m$ be as in step (3a). Let $x_m \in X$ be such that $\|x_m\| = 1$ and
 $supp(x_m) \subseteq \varphi^m(O_m)$. Finally, let
 \begin{equation}\label{eq47}
  y_m = \sum \limits_{j=0}^{2m}  \big{(}1 - \frac{1}{\sqrt{m}}\big{)}^{|m-j|} r^{-j} T^j x_m.
\end{equation}
A simple computation shows (see~\cite{Ki1}) that
\begin{equation*}
\|Ty_m - ry_m\| = o(\|y_m\|), m \rightarrow \infty,
\end{equation*}
a contradiction.
\end{proof}

\begin{remark} \label{r11}. The proof of Theorem~\ref{t15} shows that if we drop the assumption that the set of all periodic points of $\varphi$ is empty we still can claim that $|\sigma(T)| = |\sigma(S)|$.
\end{remark}

\begin{theorem} \label{t16}
  Assume conditions of Theorem~\ref{t15}. Then $\sigma_{a.p.}(T) = \sigma_{a.p.}(S)$ and $\sigma_r(T) = \sigma_r(S)$.
\end{theorem}

\begin{proof}
  It is enough to prove the first equality; the second follows from it and Theorem~\ref{t15}.
  We have already proved the inclusion $\sigma_{a.p.}(S) \subseteq \sigma_{a.p.}(T)$ in the course of proving Theorem~\ref{t15}. The proof of the inverse inclusion, $\sigma_{a.p.}(T) \subseteq \sigma_{a.p.}(S)$ repeats almost verbatim the corresponding proof in the case of Dedekind complete Banach lattices (see~\cite[Proof of Theorem 3.20, part 2]{Ki2}). Nevertheless, we include it for the reader's convenience.

   Assume that $\lambda \in \sigma_r(S)$. First notice that $\lambda \neq 0$. Indeed, if $0 \in \sigma(S)$ then there is a point $k \in K$ such that $w(k) = 0$. For any $n \in \mathds{N}$ let $O_n$ be a clopen neighborhood of $k$ such that  $|w| < 1/n$ on $O_n$ and let $x_n \in X$ be such that $\|x_n\| = 1$ and $supp x_n \subseteq O_n$. Then obviously $Tx_n \rightarrow 0$ and therefore $0 \in \sigma_{a.p.}(T)$.

    Thus, we can assume that $\lambda = 1$.
By Theorem~\ref{t35} $K$ is the union of disjoint $\varphi$-invariant sets $K_1, K_2,$ and $O$ with the properties
\begin{enumerate}[(I)]
  \item The sets $K_1$ and $K_2$ are closed in $K$.
  \item $\sigma(S, C(K_1)) \subseteq \{\lambda \in \mathds{C} : \; |\lambda| < 1\}$.
  \item $\sigma(S, C(K_2)) \subseteq \{\lambda \in \mathds{C} : \; |\lambda| > 1\}$.
  \item For any clopen subset $E$ of $O$ and for any clopen neighborhoods $V_1$, $V_2$ of $K_1$ and $K_2$, respectively there is an $m \in \mathds{N}$ such that $\varphi^m(E) \subseteq V_1$ and $\varphi^{-m}(E) \subseteq V_2$.
\end{enumerate}
Assume contrary to our claim that there are $x_n \in X$ such that $\|x_n\| = 1$ and $Tx_n - x_n \mathop \to \limits_{n \to \infty} 0$. It follows from $(IV)$ that there are a clopen neighborhood $V$ of $K_1$ and a $q \in \mathds{N}$ such that $\varphi^{(q)}(V) \subset V$. By considering operators $T^q$ and $S^q$ instead of $T$ and $S$, respectively, we can assume that $\varphi(V) \subset V$.
Then $S$ acts on $C(V)$ and by~\cite[Theorem 3.23]{Ki1} $\rho(S,C(V)) = \rho(S, C(K_1)) < 1$. Therefore there are $m \in \mathds{N}$ and $a \in (0,1)$ such that $\|w_p\|_{C(V)} \|U^p\| \leq a^p$ for $p \in \{m, m+1, \ldots \}$. Let us fix such a $p$. Then $T^p x_n - x_n \to 0$ whence $P_V T^p x_n - P_V x_n \to 0$. But $\|P_V T^p x_n\| = \|P_V w_p U^p x_n\| \leq a^p$ and therefore $\limsup \|P_V x_n\| \leq a^p$. Because $p$ can be chosen arbitrary large we see that $\limsup \|P_V x_n\| =0$. Therefore we can assume that $supp \, x_n \subseteq K_2 \cup (O \setminus V)$. But then it is not difficult to see from $(IV)$ and~(\ref{eq35})  that there are $A > 1$ and $m \in \mathds{N}$ such that $\|T^m x_n\| \geq A, \forall \; n \in \mathds{N}$ in obvious contradiction with $(T^m x_n - x_n) \mathop \to \limits_{n \to \infty} 0$.
\end{proof}

We will now discuss how Theorems~\ref{t15} and~\ref{t16} should be modified if we relinquish the condition that the set of all $\varphi$-periodic points is empty.  We start with the following situation. Let $X$ be a Kaplansky module, $U$ be a an operator from $L(X)$ that commutes with $C(K)$,  $w \in C(K)$, and $T = wU$. We want to describe the spectrum $\sigma(T)$ providing that we have the description of $\sigma(U)$. To this end for any $k \in K$ we introduce the sets $\sigma^k(U)$ and $\sigma_{a.p.}^k(U)$ as follows.

Let $\{V_\alpha\}$ be the net of all clopen neighborhoods of $k$ in $K$. Let $\chi_\alpha$ be the characteristic function of $V_\alpha$. Because $U$ commutes with $\chi_\alpha$ the operator $U_\alpha = \chi_\alpha U$ acts on $X_\alpha = \chi_\alpha X$. Let $\sigma^\alpha(U) = \sigma(U_\alpha)$ and
$\sigma_{a.p.}^\alpha(U) = \sigma_{a.p.}(U_\alpha)$. Finally, let
\begin{equation*}
  \sigma^k(U) = \bigcap \limits_\alpha \sigma^\alpha(U) \; \text{and}\; \sigma_{a.p.}^k(U) = \bigcap \limits_\alpha \sigma_{a.p.}^\alpha(U).
\end{equation*}

\begin{theorem} \label{t18} Let $X$ be a Kaplansky module, $U$ be a an operator from $L(X)$ that commutes with $C(K)$,  $w \in C(K)$, and $T = wU$. Then
\begin{equation*}
  \sigma(T) = \bigcup \limits_{k \in K} w(k)\sigma^k(U) \; \text{and} \; \sigma_{a.p.}(T) = \bigcup \limits_{k \in K} w(k)\sigma_{a.p.}^k(U).
\end{equation*}
\end{theorem}

\begin{proof} We divide the proof into four steps.

\noindent (1) We will prove that $\sigma(T) \subseteq \bigcup \limits_{k \in K} w(k)\sigma^k(U)$. Let $\lambda \in \mathds{C}$ be such that $ \lambda \not \in \bigcup \limits_{k \in K} w(k)\sigma^k(U)$. Let us fix $k \in K$. Then we can find a clopen neighborhood $V_k$ of $k$ such that
\begin{equation}\label{eq51}
  \lambda \not \in \{w(t)\alpha : t \in V_k, \alpha \in \sigma(U|\chi_{V_k}X)\}.
\end{equation}
It follows from~(\ref{eq51}) and the fact that the operators $w$ and $U$ commute that $\lambda \not \in \sigma(T|\chi_{V_k}X)$. Therefore we can find disjoint clopen subsets $V_1, \ldots , V_p$ of $K$ such that
\begin{equation}\label{eq52}
  K = \bigcup \limits_{i=1}^p V_i \; \text{and} \; \lambda \not \in \sigma(T|\chi_{V_i}X), i = 1, \ldots , p.
\end{equation}
It follows immediately from~(\ref{eq52}) that $\lambda \not \in \sigma(T)$.

\noindent (2) We will prove the inverse inclusion  $\bigcup \limits_{k \in K} w(k)\sigma^k(U)\subseteq \sigma(T)$. Assume that for some $k \in K$ we have $\lambda = w(k)\gamma$ where $\gamma \in \sigma^k(U)$. We can assume without loss of generality that $w(k) \neq 0$. Assume also, contrary to our claim, that the operator $\lambda I - T$ is invertible. For any clopen neighborhood $V_\alpha$ of $k$ we consider the function $w_\alpha \in C(K)$ defined as
\begin{equation*}
  w_\alpha(t) =\left\{
     \begin{array}{ll}
       w(k), & \hbox{$t \in V_\alpha$;} \\
       w(t), & \hbox{$t \in K \setminus V_\alpha$.}
     \end{array}
   \right.
\end{equation*}
Let $T_\alpha = w_\alpha U$. Then there is an $\alpha$ such that the operator $\lambda I - T_\alpha$ is invertible. Therefore, the operator $(\lambda I - T_\alpha)\chi_\alpha$ is invertible on $\chi_\alpha X$. But  $(\lambda I - T_\alpha)\chi_\alpha = w(k)(\gamma I - U)\chi_\alpha$ whence the operator $(\gamma I - U)\chi_\alpha$ is invertible on $\chi_\alpha X$, a contradiction.

\noindent (3) On this step we will prove the inclusion  $\bigcup \limits_{k \in K} w(k)\sigma_{a.p.}^k(U)\subseteq \sigma_{a.p.}(T)$. Let $\lambda = w(k)\gamma$ where $\gamma \in \sigma_{a.p.}^k(U)$. Let us fix an $\varepsilon > 0$ and a clopen neighborhood $V$ of $k$ such that
$|w(t) - w(k)| \leq \varepsilon , t \in V$. By the definition of $\sigma_{a.p.}^k(U)$ there is an $x \in \chi_v X$ such that $\|x\| = 1$ and $\|Ux - \gamma x\| \leq \varepsilon$. Then $\|Tx - \lambda x\| \leq \varepsilon(\|U\| + \|w\|)$.

\noindent (4) It remains to prove the inclusion $ \sigma_{a.p.}(T) \subseteq \bigcup \limits_{k \in K} w(k)\sigma_{a.p.}^k(U)$. Let $\lambda \not \in \bigcup \limits_{k \in K} w(k)\sigma_{a.p.}^k(U)$. Then for every $k \in K$ we can find a clopen neighborhood $V_k$ of $k$ and a constant $c(k) > 0$ such that
$\|Tx - \lambda x\| \geq c(k)\|x\|, x \in \chi_{V_k}X$. By taking a finite covering of $K$ consisting of such neighborhoods we see that the operator $\lambda I - T$ is bounded from below.
\end{proof}

\begin{corollary} \label{c5} Let $X$ be a Kaplansky module, $U$ be a $d$-isomorphism of $X$, and $\varphi$ be the corresponding homeomorphism of $K$. Assume that there is an $m \in \mathds{N}$ such that for every $k \in K$ we have $\varphi^{(m)}(k) = k$ and $\varphi^{(i)}(k) \neq k$ for every positive integer $i < m$. Let $w \in C(K)$ and $T = wU$. Then
\begin{equation*}
  \sigma(T) = \{\lambda \in \mathds{C} : \lambda^m \in \bigcup \limits_{k \in K} w_m(k)\sigma^k(U^m)\}
\end{equation*}
  and
\begin{equation*}
  \sigma_{a.p.}(T) = \{\lambda \in \mathds{C} : \lambda^m \in \bigcup \limits_{k \in K} w_m(k)\sigma_{a.p.}^k(U^m)\}.
\end{equation*}
\end{corollary}

\begin{proof} The proof follows from Theorem~\ref{t18} (applied to the operator $T^m$) and the fact that the operators $T$ and $\gamma T$ are similar for every $\gamma \in \mathds{C}$ such that $\gamma^m = 1$. See the proof of part (1) of Theorem~\ref{t11}.
  \end{proof}

Let $X$ be a Kaplansky module and $K$ be the corresponding extremally disconnected compact Hausdorff space. Let $\varphi$ be a homeomorphism of $K$ onto itself. We consider the following subsets of $K$.
\begin{itemize}
  \item $\Pi_m, m \in \mathds{N}$, is the set of all $\varphi$-periodic points of the period $m$.
  \item $\Pi = cl \bigcup \limits_{m=1}^\infty \Pi_m$.
  \item $N\Pi = K \setminus \Pi$.
\end{itemize}
By Frolik's theorem all the sets introduced above are clopen in $K$ (some of them, of course, can be empty). We will denote by $\chi_m, m \in \mathds{N}$, $\chi_\Pi$, and $\chi_{N\Pi}$ the corresponding characteristic functions.

Let $U$ be a $d$-isomorphism of $X$ and $\varphi$ be the corresponding homeomorphism of $K$. Clearly $U$ commutes with the above characteristic functions (considered as operators on $X$).  Let $X_m = \chi_m X, m \in \mathds{N}$, $X_\Pi = \chi_\Pi X$, and $X_{N\Pi} = \chi_{N\Pi} X$.

Let $w \in C(K)$ and $T = wU$. Let $\sigma_m = \sigma(T|X_m), m \in \mathds{N}$.

\begin{theorem} \label{t17} Let $X$ be a Kaplansky module, $U$ be a $d$-isomorphism of $X$, and $w \in C(K)$. Assume that $\sigma(U) \subseteq \mathds{T}$. Let $T = wU$. Then
\begin{enumerate}
  \item $\sigma_{a.p.}(T|X_m) = \sigma_m, m \in \mathds{N}$.
  \item $\sigma(T|X_\Pi) = \sigma_{a.p.}(T|X_\Pi) = cl \bigcup \limits_{m=1}^\infty \sigma(T|X_m) \cup \sigma$ where $\sigma$ is a rotation invariant subset of $\mathds{C}$.
   \item $\sigma(T|X_{N\Pi}) = \sigma(S,C(N\Pi))$ and $\sigma_{a.p.}(T|X_{N\Pi}) = \sigma_{a.p.}(S,C(N\Pi))$ where
      \begin{equation*}
        Sf = w(f \circ \varphi), f \in C(N\Pi).
      \end{equation*}
\end{enumerate}
\end{theorem}

\begin{proof} (1) follows from Theorem~\ref{t18}, Corollary~\ref{c5}, and the condition $\sigma(U) \subseteq \mathds{T}$.

\noindent (2) The fact that $\sigma$ is rotation invariant follows immediately from part (1) of Lemma~\ref{l2}. Assume, contrary to our claim that there is a $\lambda \in \mathds{C}$ such that $\lambda \in \sigma_r(T|X_\Pi)$. It follows from (1) and Corollary~\ref{c5} that there is an $m \in \mathds{N}$, such that $\lambda \mathds{T} \subseteq \sigma_r(T|X^{(m)})$ where $X^{(m)} = X_\Pi \ominus \sum \limits_{j=1}^m \oplus X_m$. By Remark~\ref{r11} there is a $\beta \in \mathds{C}$ such that $|\beta| = |\lambda|$ and $\beta \in \sigma(S,C(K_\Pi^m))$ where $K_\Pi^m = cl \bigcup \limits_{j=m}^\infty \Pi_j$. Notice that $\beta \in \sigma_{a.p.}(S,C(K_\Pi^m))$. Indeed, otherwise by Theorem~\ref{t35} there is an open nonempty subset $E$ of $K_\Pi^m$ such that the sets $\varphi^{(n)}(E), n \in \mathds{Z}$ are pairwise disjoint. This clearly contradicts the definition of $K_\Pi^m$. Thus, there is a $k \in K_\Pi^m$ such that inequalities~(\ref{eq41}) hold. it is immediate to see that $k$ cannot be a $\varphi$-periodic point. Then, the same reasoning as in the proof of Theorem~\ref{t15} shows that $\lambda \mathds{T} \subseteq \sigma_{a.p.}(T)$.

\noindent (3) follows from Theorems~\ref{t15} and~\ref{t16}.
\end{proof}

As the next example shows, we cannot claim in statement (2) of Theorem~\ref{t17} that
$\sigma(T|X_\Pi) =  cl \bigcup \limits_{m=1}^\infty \sigma(T|X_m)$.

\begin{example} \label{e2} Let $E$ be the subset of $\mathds{C}$ defined as follows.
\begin{equation*}
  E = \{(i, 1/n): n \in \mathds{N}, -2n \leq i \leq 2n\} \cup \{(i, 0): i \in \mathds{Z} \}.
\end{equation*}
  Let the map $\varphi : E \rightarrow E$ be defined as
\begin{multline*}
  \varphi((i, 0)) = (i+1, 0), i \in \mathds{Z}; \\
 \varphi((i, 1/n)) = (i+1, 1/n), n \in \mathds{N}, -2n \leq i < 2n; \\
 \varphi((2n, 1/n)) = (-2n, 1/n), n \in \mathds{N}.
\end{multline*}
Let us also define the weight $w$ as follows.
\begin{equation*}
w(i, 1/n) =  \left\{
    \begin{array}{ll}
      1, & \hbox{if $-2n \leq i < -n$;} \\
      2, & \hbox{if $-n \leq i \leq n$;} \\
      1
, & \hbox{if $n < i < 2n$.}
    \end{array}
  \right.
 \end{equation*}
\begin{equation*}
  w(i,0) = 2, i \in \mathds{Z}.
\end{equation*}
Clearly $E$ endowed with the topology inherited from $\mathds{C}$ is a locally compact space, $\varphi$ is a homeomorphism of $E$, and $w$ is continuous and bounded on $E$. Let $\beta E$ be the Stone -  \v{C}ech compactification  of $E$. Then $\varphi$ extends uniquely to a homeomorphism $\tau$ of $\beta E$ and $w$ can be identified with the unique function $W \in C(\beta E)$. We define the disjointness preserving operator $T$ on $C(\beta E)$ as
\begin{equation*}
  Tf = W(f \circ \tau), f \in C(\beta E).
\end{equation*}
Notice that $T$ is order continuous on $C(\beta E)$.

Let $K$ be the absolute of $\beta E$. Then $C(K)$ can be identified with the Dedekind completion of $C(\beta E)$, the homeomorphism $\tau$ can be extended in the unique way (see~\cite{Ve}) to a homeomorphism $\psi$ of $K$, $W$ can be identified with $\hat{W} \in C(K)$, and the operator $T$ being order continuous extends in the unique way to the operator $\hat{T}$ on $C(K)$ defined as
\begin{equation*}
  \hat{T}f = \hat{W}(f \circ \psi).
\end{equation*}
Let $\Pi_m, m \in \mathds{N}$ be the set of all $\psi$-periodic points in $K$ of the period $m$.  It is easy to see that
\begin{enumerate}
  \item $\Pi_{m} = \emptyset$ if $m \neq 4k+1$.
  \item $\Pi_{4k+1}$ can be identified with the set $\{(i, 1/k): -2k \leq i \leq 2k \}$.
\item $K = cl \bigcup \limits_{k=1}^\infty \Pi_{4k+1}$.
  \item $ \sigma_{4k+1} = \sigma(\hat{T}, C(\Pi_{4k + 1})) = \{\rho \gamma : \rho = 2^{\frac{2k+1}{4k+1}}, \gamma^{4k+1} = 1\}$.
  \item $2\mathds{T} \subseteq \sigma(\hat{T})$.
\end{enumerate}
It follows from (4) and (5) that
\begin{equation*}
  cl \bigcup \limits_{k=1}^\infty \sigma_{4k+1} \subsetneqq \sigma(\hat{T}).
\end{equation*}
\end{example}

We turn now to essential spectra of operators of the form $wU$.

\begin{theorem} \label{t21} Let $X$ be a Kaplansky module, $U$ be a $d$-isomorphism of $X$, $w \in C(K)$, $\sigma(U) \subseteq \mathds{T}$, and $T = wU$. Let $\varphi$ be the homeomorphism of $K$ corresponding to the automorphism of $C(K)$: $f \rightarrow UfU^{-1}$. Assume that the set of all $\varphi$-periodic points in $K$ is empty and also assume that $K$ has no isolated points. Then
\begin{equation*}
  \sigma_2(T) = \sigma_{a.p.}(T).
\end{equation*}
\end{theorem}

\begin{proof}
  The inclusion $\sigma_2(T) \subseteq \sigma_{a.p.}(T)$ follows from definition of $\sigma_2(T)$. Assume that $\lambda \in \sigma_{a.p.}(T)$. First assume that $\lambda = 0$. Then $w$ is not invertible in $C(K)$. Let $k \in K$ be such that $w(k) = 0$. Because $k$ is not isolated in $K$ we can find clopen pairwise disjoint subsets $V_n, n \in \mathds{N}$ of $K$ such that $|w(t)| < 1/n, t \in V_n$. Let $x_n \in X$ be such that $\|x_n\| = 1$ and $\chi_{V_n}x_n = x_n$. Then clearly $\|Tx_n\| \leq (1/n)\|U\|$. Moreover, because the elements $x_n$ have pairwise disjoint supports the sequence $x_n$ is singular.

Let now $\lambda \in \sigma_{a.p.}(T)$, $\lambda \neq 0$. By Theorem~\ref{t16} and~\cite[Lemma 3.25]{Ki1} there is a point $k \in K$ such that the inequalities~(\ref{eq41}) hold. We proceed now almost as in the proof of Theorem~\ref{t16}. Namely, because $k$ is neither an isolated nor a $\varphi$-periodic point we can find clopen nonempty subsets $O_m, m \in \mathds{N}$ of $K$ such that
\begin{itemize}
  \item The sets $\varphi^{(j)}(O_m), |j| \leq m+1, m \in \mathds{N}$, are pairwise disjoint.
  \item On set $O_m$ inequalities~(\ref{eq48}) hold.
  \end{itemize}
Let $y_m$ be as in the proof of Theorem~\ref{t16} and $z_m = y_m/\|y_m\|$. Then $Tz_m - \lambda z_m \rightarrow 0$ and the sequence $z_m$ is singular.
\end{proof}

\begin{theorem} \label{t30} Let $X$ be a Kaplansky module, $U$ be a $d$-isomorphism of $X$, $w \in C(K)$, $\sigma(U) \subseteq \mathds{T}$, and $T = wU$. Let $\varphi$ be the homeomorphism of $K$ corresponding to the automorphism of $C(K)$: $f \rightarrow UfU^{-1}$. Assume that the set of all $\varphi$-periodic points in $K$ is empty and also assume that $K$ has no isolated points. Consider operator $\tilde{T} = wU^{-1}$. Then
  \begin{equation*}
    \sigma_2(T^\prime) = \sigma_{a.p.}(\tilde{T}).
  \end{equation*}
\end{theorem}

\begin{proof}
  We prove first that $\sigma_{a.p.}(\tilde{T}) \subseteq \sigma_2(T^\prime)$. Let $\lambda \in \sigma_{a.p.}(\tilde{T})$. The case $\lambda = 0$ can be considered in the same way as in the proof of Theorem~\ref{t21}. Assume therefore that $\lambda \neq 0$. From Theorem~\ref{t16} and~\cite[Lemma 3.25]{Ki1} we can conclude that there is $k \in K$ such that
\begin{equation}\label{eq53}
  |w_n(k)| \leq r^n, \; |w_n(\varphi^{(-n)}(k)| \geq r^n, \; n \in \mathds{N}.
\end{equation}
The map $f \rightarrow f^\prime, f \in C(K)$ defines on $X^\prime$ the structure of a Banach $C(K)$-module. \footnote{Of course, in general $(X',C(K))$ is not a Kaplansky module.}  Notice that for any $f \in C(K)$ we have
$(f \circ \varphi^{-1})' = (U^{-1}fU)' = U'f'(U')^{-1}$ and therefore to the map $f' \rightarrow U'f'(U')^{-1}$ corresponds the homeomorphism $\varphi^{(-1)}$ of $K$. Now, by using inequalities~(\ref{eq53}) we can prove that $\lambda \in \sigma_2(T')$ similarly to the proof of Theorem~\ref{t21}.

Next we need to consider the case when $\lambda \in \sigma_r(\tilde{T})$. It follows from Theorem~\ref{t16} and from~\cite[Theorem 3.29]{Ki1} that there are closed disjoint subsets $K_1, K_2$ and $O$ of $K$ with the following properties.
\begin{enumerate}
  \item $\varphi(K_1) = K_1$ and $\rho(S, C(K_1) < |\lambda|$.
  \item $\varphi(K_2) = K_2$. The operator $S$ is invertible on $C(K_2)$ and $\rho(S^{-1}, C(K_2) < |\lambda|$.
  \item The set $O$ is clopen in $K$, the sets $\varphi^{(j)}(O), j \in \mathds{Z}$ are pairwise disjoint, and $K =K_1 \cup K_2 \cup \bigcup \limits_{j= - \infty}^\infty \varphi^{(j)}(O)$.
  \item $\bigcap \limits_{n=1}^\infty cl \bigcup \limits_{j=n}^\infty \varphi^{(j)}(O) \subseteq K_2 $.
  \item $\bigcap \limits_{n=1}^\infty cl \bigcup \limits_{j=n}^\infty \varphi^{(-j)}(O) \subseteq K_1 $.
\end{enumerate}
Assume that there is a sequence $F_n \in X'$ such that $\|F_n\| = 1$ and $T'F_n - \lambda F_n \rightarrow 0$. We can bring this assumption to a contradiction using properties (1) - (5) above and the same reasoning as in the proof of Theorem~\ref{t16}.
\end{proof}

\begin{corollary} \label{c7}
  Assume conditions of Theorem~\ref{t22}. Then
\begin{equation*}
  \sigma(T) = \sigma_3(T).
\end{equation*}
\end{corollary}

\begin{corollary} \label{c6} Assume conditions of Theorem~\ref{t22}. Assume also that $K$ contains no clopen $\varphi$-wandering subset, i.e. for every nonempty clopen $O \subseteq K$ the sets $\varphi^{(j)}(O), j \in \mathds{Z}$ cannot be pairwise disjoint. Then
\begin{equation*}
  \sigma(T) = \sigma_{a.p.}(T) = \sigma_1(T).
\end{equation*}

\end{corollary}

Let us look now what happens when we drop the condition that $\varphi$ has no periodic points but still assume that $K$ has no isolated points. Combining Theorems~\ref{t17} and~\ref{t22} with Theorem 3.29 from~\cite{Ki1} we have the following (using the same notations as in the statement of Theorem~\ref{t17}).

 \begin{theorem} \label{t31} Let $X$ be a Kaplansky module, $U$ be a $d$-isomorphism of $X$ such that
   $\sigma(U) \subseteq \mathds{T}$, $\varphi$ be the corresponding homeomorphism of $K$, $w \in C(K)$ and $T = wU$. Assume that $K$ contains no isolated points. Then
\begin{enumerate}
  \item $\sigma(T) = \sigma_3(T)$.
  \item $\sigma_1(T) = \sigma_1(T|X_{N\Pi}) \cup \sigma(T|X_\Pi)$.
 \item $\sigma_2(T) = \sigma_2(T|X_{N\Pi}) \cup \sigma(T|X_\Pi)$.
 \item $\sigma_2(T') = \sigma_2(T'|(X_{N\Pi})') \cup \sigma(T|X_\Pi)$.
  \item If $K$ contains no clopen $\varphi$-wandering subset then $\sigma(T) = \sigma_1(T)$.
\end{enumerate}
 \end{theorem}

To complete our description of essential spectra of operators of the form $wU, \sigma(U) \subseteq \mathds{T}$ we need to get rid of the condition that $K$ contains no isolated points. The following theorem provides a description of the sets $\sigma_3(T)$ and $\sigma_4(T)$. The proof of this theorem is simple and resembles the proof of the corresponding result for Dedekind complete Banach lattices in~\cite{Ki3}. Therefore we omit it.

\begin{theorem} \label{t24} Let $X$ be a Kaplansky module, $U$ be a $d$-isomorphism of $X$ such that
   $\sigma(U) \subseteq \mathds{T}$, $\varphi$ be the corresponding homeomorphism of $K$, $w \in C(K)$, and $T = wU$. Assume that $\lambda \in \sigma(T)$. The operator $\lambda I - T$ is Fredholm if and only if there are subsets $K_1, K_2, K_3, K_4$ of $K$ such that
\begin{enumerate}
  \item The sets $K_1, \ldots ' K_4$ are at most finite and at least one of sets $K_1, K_2, K_3$ is not empty.
  \item If $k \in K_1 \cup K_2 \cup K_3$ then $k$ is an isolated point of $K$ and $\dim \chi_{\{k\}}X = n_k < \infty$.
\item If $k_1, k_2 \in K_1 \cup K_2 \cup K_3$ then $\varphi^{(i)}(k_1) \neq \varphi^{(j)}(k_2), i, j \in \mathds{Z}$.
  \item Let $k \in K_1$, $L_1 = \bigcap \limits_{n=1}^\infty cl\{\varphi^{(j)}(k) : k \geq n\}$ and
$L_2 = \bigcap \limits_{n=1}^\infty cl\{\varphi^{(-j)}(k) : k \geq n\}$. Then $\rho(S, C(L_2)) < |\lambda|$, the operator $S$ is invertible on $C(L_1)$ and $\rho(S^{-1},C(L_1)) < 1/|\lambda|$.
  \item If $k \in K_2$ and $L_1$ and $L_2$ are as above then $\rho(S, C(L_1)) < |\lambda|$, the operator $S$ is invertible on $C(L_2)$ and $\rho(S^{-1},C(L_2)) < 1/|\lambda|$.
  \item If $k \in K_3$ then $k$ is a $\varphi$-periodic point and $\lambda^p \in w_p(k)\sigma(U^p|\chi_{\{k\}}X)$.
\item  If $k \in K_4$ then $k$ is a $\varphi$-periodic point, $\lambda^p \not \in w_p(k)\sigma(U^p|\chi_{\{k\}}X)$, but $|\lambda^p|  \in |w_p(k)||\sigma(U^p|\chi_{\{k\}}X)|$.
\item $|\lambda| \not \in |\sigma(S, C(K \setminus \bigcup \limits_{i=1}^4 K_i)|$.
\end{enumerate}
Moreover
\begin{equation*}
  ind(\lambda I - T) = \sum \limits_{k \in K_1} n_k - \sum \limits_{k \in K_2} n_k.
\end{equation*}
  \end{theorem}

  \begin{remark} \label{r22} If we assume that $X$ is a finitely generated Kaplansky module then the condition $\dim{\chi_{\{k\}}X} < \infty$ in part (2) of the statement of Theorem~\ref{t24} is satisfied automatically.
  \end{remark}

\begin{corollary} \label{c8} Let $X$ be a Kaplansky module, $U$ be a $d$-isomorphism of $X$ such that
   $\sigma(U) \subseteq \mathds{T}$, $w \in C(K)$ and $T = wU$. Then
\begin{equation*}
  \sigma_5(T) = \sigma_4(T).
\end{equation*}
\end{corollary}

\begin{corollary} \label{c9} Let $X$ be a Kaplansky module, $U$ be a $d$-isomorphism of $X$ such that
   $\sigma(U) \subseteq \mathds{T}$, $w \in C(K)$ and $T = wU$. The following conditions are equivalent.
\begin{enumerate}
  \item $\lambda \in \sigma(T) \setminus \sigma_5(T)$.
  \item The operator $\lambda I - T$ is Fredholm and the sets $K_1$ and $K_2$ from the statement of Theorem~\ref{t24} are empty.
\end{enumerate}
\end{corollary}

\section{Examples.}

\subsection{Spaces of vector valued continuous functions.}

Let $K$ be a compact Hausdorff space and $Y$ be a Banach space. We consider the Banach space $C(K, Y)$ of  all continuous maps from $K$ into $Y$ with the standard norm
\begin{equation*}
  \|x\| = \max \limits_{k \in K} \|x(k)\|_Y.
\end{equation*}
It is immediate to see that $C(K, Y)$ is a Banach $C(K)$-module with the Fatou property.
Let $E$ be an open nonempty subset of $K$, $\varphi$ be a continuous map from $E$ into $K$, and $w$ be a continuous map from $E$ into $L(Y)$ such that
\begin{equation*}
  \sup \limits_{k \in E} \|w(k)\| < \infty .
\end{equation*}
The formula
\begin{equation}\label{eq67}
(Tx)(k) =  \left\{
    \begin{array}{ll}
      w(k)x(\varphi(k)), & \hbox{$k \in E$;} \\
      0, & \hbox{otherwise}
    \end{array}
  \right.
\end{equation}
defines a bounded $d$-endomorphism of $C(K,Y)$. Moreover, if the set of all eventually $\varphi$-periodic points is of first category in $K$ then the powers of $T$ are $d$-independent and we have the following corollary of Theorem~\ref{t9}.

\begin{corollary} \label{c17} Let $T$ be an operator on $C(K,Y)$ defined by the formula~(\ref{eq67}) and let the set of all eventually $\varphi$-periodic points be of first category in $K$. Then the sets
\begin{equation*}
  \sigma(T), \sigma_{a.p.}(T), \; \text{and} \; \sigma_2(T)
  \end{equation*}
 are rotation invariant.
\end{corollary}

\begin{remark} \label{r23} In the case when $Y = \mathds{C}$ it is well known (see e.g.~\cite[Theorem 3.1]{Ki1}) that every continuous disjointness preserving operator on $C(K)$ is of the form~(\ref{eq67}). Whether a similar statement is correct when $\dim{Y} \geq 2$ is, to the best of our knowledge, an open question. We know nevertheless (see~\cite[pp 128 - 129]{AAK}), that if $T$ is a continuous $d$-endomorphism of $C(K,E)$ then there are an open subset $E_T$ of $K$ and a continuous map $\varphi_T : E_T \rightarrow K$ such that if the set of all eventually $\varphi$-periodic points is of first category in $K$, then the powers of $T$ are $d$-independent, and therefore the conclusion of Corollary~\ref{c17} remains valid.
\end{remark}

Assume now that $K$ is an extremally disconnected compact space. Then $C(K,Y)$ is a Kaplansky module with the Fatou property and we can apply the corresponding results from Sections 3 and 4. We will also state the following analog of Theorem~\ref{t18}.

\begin{theorem} \label{t38} Let $K$ be an extremally disconnected compact space. Let $Y$ be a Banach space and $w : K \rightarrow L(Y)$ be a continuous map. Let $T \in L(C(K,Y))$ be defined as
\begin{equation*}
(Tx)(k) = w(k)x(k), x \in C(K,Y), k \in K.
\end{equation*}
Then
\begin{equation*}
  \sigma(T) = \bigcup \limits_{k \in K} \sigma(w(k)) \; \text{and} \; \sigma_{a.p.}(T) = \bigcup \limits_{k \in K} \sigma_{a.p.}(w(k)).
\end{equation*}
 \end{theorem}

 The proof of Theorem~\ref{t38} is similar to that of Theorem~\ref{t18} and we omit it.

 \begin{corollary} Assume conditions of Theorem~\ref{t38}. Assume additionally that for every $k \in K$ we have $\sigma(w(k)) = \sigma_{a.p.}(w(k))$. Then
 \begin{equation*}
   \sigma(T) = \sigma_{a.p.}(T)
 \end{equation*}
 Moreover, if $K$ has no isolated points then
 \begin{equation*}
   \sigma(T) = \sigma_2T)
 \end{equation*}
   \end{corollary}

   \begin{remark} \label{r24} The condition $\sigma(w(k)) = \sigma_{a.p.}(w(k)), k \in K$ is satisfied, in particular, in the following two important cases
   \begin{itemize}
     \item $\dim{Y} < \infty$.
     \item $\forall k \in K$ the operator $w(k)$ is a surjective isometry on $Y$.
   \end{itemize}
   \end{remark}
 We leave it to the reader to state the corresponding analog of Corollary~\ref{c5}

\noindent Finally let us notice that if $w(k) = f(k)u(k)$, where $f \in C(K)$ and $\forall k \in K  \; \sigma(u(k)) \subseteq \mathds{T}$ then $T = fU$, $\sigma(U) \subseteq \mathds{T}$, and we can apply the corresponding results of Section 5.

\subsection{K\"{o}the-Bochner function spaces.}
We will not try to discuss this example in full generality, but will consider only a comparatively simple situation which, nevertheless, provides a good illustration of our previous results.

 Let $(\Omega, \Sigma, \mu)$ be a space with a finite complete non-atomic measure $\mu$ and $\psi : \Omega \rightarrow \Omega$ be a measure preserving ergodic transformation. \footnote{I.e. $\psi^{-1}(A) = A$ if and only if either $\mu(A) = 0$ or $\mu(A) = \mu(\Omega)$.} Let $I$ be a Banach ideal in the vector lattice $L^0(\Omega, \Sigma, \mu)$ of (classes) all $\mu$-measurable functions on $\Omega$. We assume that the operator $x \rightarrow x \circ \psi$ acts on $I$ and that $\|x \circ \psi\| = \|x\|, x \in I$. This condition will be satisfied if, for example, $I$ is an interpolation space between $L^\infty(\Omega, \Sigma, \mu)$ and $L^1(\Omega, \Sigma, \mu)$. Let $Y$ be a Banach space and $X$ be the space of all strongly measurable $Y$-valued functions on $\Omega$ such that
 \begin{equation*}
   x \in X \Leftrightarrow \ell(\cdot) =  \|x(\cdot)\|_y \in I.
 \end{equation*}
Endowed with the norm $\|x\| = \|\ell(\cdot)\|_Y$ $X$ becomes a Banach space (see e.g.~\cite{Pe}).
It is easy to see that $X$ is a Kaplansky $C(K)$-module, where $K$ is the Gelfand compact of the Banach algebra $L^\infty(\Omega, \Sigma, \mu)$. Let $w: \Omega \rightarrow L(Y)$ be a strongly measurable (see~\cite[p.162]{Pe}) essentially bounded function. Then the formula
\begin{equation*}
  (Tx)(\omega) = w(\omega)x(\psi(\omega))
\end{equation*}
defines a bounded $d$-endomorphism of $X$. Let $\varphi$ be the corresponding map from Proposition~\ref{p3}. It is immediate that $\varphi$ is an open map of $K$ onto itself and that the set of all $\varphi$-periodic points is of first category in $K$. Moreover, because $K$ is hyperstonean, the set of all $\varphi$-periodic points is empty. By Theorem~\ref{t2} the sets $\sigma(T), \sigma_{a.p.}(T), \sigma_r(T)$ as well as the essential spectra $\sigma_i(T), i = 1, \ldots , 5$ and $\sigma_2(T')$ are rotation invariant.

As a simple example we can consider the space $(\mathds{T}, \Sigma, m)$, where $m$ is the Lebesgue measure, and the ergodic transformation $\psi(\omega) = \omega^2, \omega \in \mathds{T}$.

\section{Appendix. Properties of Disjointness Preserving Operators }
\label{PDPO}
In this section we will review the properties of disjointness preserving operators that we used throughout the paper. Initially we will assume that $X, Y$ are complex Banach lattices and $T:X\rightarrow Y$ is a (bounded) disjointness preserving operator. That is $x\bot y \Rightarrow T(x)\bot T(y)$ for all $x,y\in X.$

The properties of disjointness preserving operators on real Banach lattices and more generally those of order bounded disjointness preserving operators on real Riesz spaces are well-known. In~\cite{AAK} these properties are developed in even more generality on complex Riesz spaces for order bounded disjointness preserving operators. However this is not entirely routine and some details in~\cite{AAK} are missing. We will present a complete discussion of the passage to the complex case for Banach lattices.

Let $X$ be a complex Banach lattice. Then there exists a real Banach lattice $X_{r} = X_{+}-X_{+}\subseteq X$ such that $X_{r}\cap iX_{r}=\left\{0\right\}$ and $X=X_{r}\oplus iX_{r}$. Moreover, let $z=x+iy \in X$ with $x,y\in X_{r}$ then
\begin{multline*}
 |z|=sup\left\{xcos\theta+ysin\theta:\theta \in [0,2\pi]\right\}= \\ =sup\left\{|x|cos\theta+|y|sin\theta:\theta \in [0,\pi/2]\right\}
 \end{multline*}
and $\left\|z\right\|=\left\||z|\right\|.$ When $X$ is represented as a Banach lattice of functions with respect to pointwise operations on some set (as it always can be), it is familiar that the above formula yields the usual
\begin{equation*}
   |z|=\sqrt{x^2+y^2}.
\end{equation*}
For each operator $T:X \rightarrow Y$ there exist operators $T_r$ and $T_i$ from $X_r$ into $Y_r$ such that for each $z=x+iy\in X$ with $x,y\in X_r$, one has
\begin{equation*}
  T(z)=T_r (x)- T_i (y)+i(T_i (x)+T_r (y)).
\end{equation*}
 Suppose that $T$ is a disjointness preserving operator. Then $T_r , T_i$ are disjointness preserving operators from $X_r$ into $Y_r$, since for each $x\in X_r,$ one has that both $|T_r (x)|, |T_i (x)|$ are less than or equal to $|T(x)|$ in the partial order of $Y_r.$

In the case of real Banach lattices, Abramovich~\cite{Ab} proved that disjointness preserving operators are order bounded, and therefore by a result of Meyer~\cite{Me}, they are regular. Elementary proofs of these results were given by de Pagter~\cite{Pa} and Bernau~\cite{Be}, respectively (See also~\cite[Thorems 3.1.4 and 3.1.5]{Mn}). An operator $T:X\rightarrow Y$ is by definition regular if $T_r, T_i$ are regular in the usual sense. So our final observation in the above paragraph shows that if $T$ is disjointness preserving, then $T$ is regular and order bounded. (In the sense that: Given $z\in X$ there exists $w\in Y_+$ such that, for any $u\in X$, $|u|\leq |z|$ implies that $|T(u)|\leq w$.)

When $X$ is a Banach lattice we will denote by $\widehat{X}$, the Dedekind completion of $X$ where $\widehat{X}=\widehat{X}_r\oplus i\widehat{X}_r.$(Here $\widehat{X}_r$ is the Dedekind completion of $X_r$.) It is familiar that there exists a stonean compact Hausdorff space $Q(X)$ such that $\widehat{X}$ is an ideal in $C_\infty (Q(X))$ and $Z(\widehat{X})=C(Q(X))$. As usual $C_\infty (Q(X))$ denotes the algebra and the Dedekind complete Riesz space of all continuous functions $f:C(Q(X))\rightarrow  \mathds{C}\cup \left\{\infty\right\}$ (where $\mathds{C}\cup \left\{\infty\right\}$ is the one point compactification of $\mathds{C}$) such that $f^{-1}(\infty)$ is a nowhere dense subset of $Q(X).$

By the support of a set $A\subset \widehat{X},$ with a slight abuse of language, we will denote $\chi_A,$ the band projection of $\widehat{X}$ onto $A^{dd}.$ We may also think of $\chi_A$ as the characteristic function of the clopen set in $Q(X)$ that is the closure of all points $t\in Q(X)$ such that some element of $A$ takes a finite non-zero value at $t.$ Given an operator $T:X\rightarrow Y,$ we will denote by $\chi_T$ the support of the range of $T$ in $\widehat{Y}.$

Now we can state and prove the main result for disjointness preserving operators on complex Banach lattices in~\cite{AAK}.
 \begin{theorem} \label{at1}
 Let $X,Y$ be Banach lattices and $T:X\rightarrow Y$ be an operator. Then the following are equivalent:
\begin{enumerate}
\item T is a disjointness preserving operator.
\item For all $x,y\in X$, $|y|\leq |x|$ implies that $|T(y)|\leq |T(x)|.$
\item T admits a unique factorization $T=VT_l$ such that $T_l:X\rightarrow Y$ is a lattice homomorphism and $V\in Z(\widehat{Y})$ such that $|V|=\chi_T.$
\end{enumerate}
\end{theorem}
\begin{proof} Consider $(1)\Rightarrow (2).$ Given any $x\in X,$ we may assume that $I(x)$ (the ideal generated by $|x|$ in $X$) is lattice isomorphic to $C(K_x)$ for some compact Hausdoff space $K_x.$ Moreover, by our observations preceding the statement of the theorem, $T$ is order bounded. Hence we may assume that $T(I(x))$ is contained in an ideal $I(w)$ in $Y$ for some $w\in Y_+.$ So $I(w)$ is also lattice isomorphic to $C(K_w)$ for some compact Hausdorff space $K_w$. Now, since $T$ is order bounded, when restricted to $I(x)$, as a linear mapping on $C(K_x)$ into $C(K_w)$, $T$ is bounded with respect to the sup-norms on these spaces. For any fixed $t\in K_w$, consider the bounded linear functional $l:C(K_x)\rightarrow \textbf{C}$ defined by $l(y)=T(y)(t).$ To establish $(1)\Rightarrow (2)$, it is sufficient to show that $|l(y)|\leq |l(x)|$ whenever $|y|\leq |x|$ for some $y\in X.$ Let the closed set $F\subset K_x$ be the support of the regular Borel measure $l\in C(K_x)^{'}.$ It is clear that the result would be established if we show that $F$ is a singleton.

Suppose that $F$ is not a singleton. Suppose that $l$ is a positive linear functional and $s_1, s_2$ are distinct points in $F.$ We can find Urysohn functions $y_1,y_2\in C(K_x)$ such that $y_1(s_1)=1, y_2(s_2)=1$ and $y_1\bot y_2$. Then $l(y_1)l(y_2)>0$. But that $T(y_1)\bot T(y_2)$ implies that at least one of them should be zero. This yields a contradiction. Suppose $l$ is real. Then we may suppose that $F_+,F_-$ are the non-empty the supports of $l^+, l^-$ respectively such that $F=F_+\cup F_-$ and $F_+\cap F_-=\emptyset.$ If we have $s_1\in F_+$ and $s_2\in F_-$, again we may find Urysohn functions $y_1,y_2\in C(K_x)$ such that $y_1\bot y_2$ and $l(y_1)>0>l(y_2)$. Since $T(y_1)\bot T(y_2)$, this still leads to a contradiction. Now if $l$ is a complex measure and, without loss of generality the support of its real part is not a singleton, we can repeat the above argument to find two functions $y_1,y_2\in C(K_x)$ such that $y_1\bot y_2$ and both $real(l(y_1))$ and $real(l(y_2))$ are non-zero. This means that $l(y_1)\neq 0$ and $l(y_2)\neq 0$ with $T(y_1)\bot T(y_2)$, and leads to a contradiction. Hence we are reduced to the case where the supports of the real part and the imaginary part of $l$ are distinct singletons. Obviously that also will yield a contradiction. Therefore $F= \left\{s_t \right\}$ for some $s_t\in K_x$. This proves $(1)\Rightarrow(2).$

The converse $(2)\Rightarrow (1)$ is almost immediate. It is well known for real Banach lattices that for any $x,y\in X$ we have
\begin{equation} \label{eq70}
 x\bot y \Leftrightarrow |y|\leq |y+\lambda x|
\end{equation}
for all scalars $\lambda$. But the same is true for complex Banach lattices as well when $\lambda \in \textbf{C}$. (Namely when $x,y$ are complex numbers and $\lambda \in \textbf{C}$ then it is easy to see that~(\ref{eq70}) is true. To see~(\ref{eq70}) for complex Banach lattices, simply think of the Banach lattice represented as a lattice of functions on some base space and use the preceding fact at each point of the base space.) So if $x\bot y$ in $X$, we have, by~(\ref{eq70}) and $(2)$, $|T(y)|\leq |T(y)+\lambda T(x)|$ for all $\lambda \in \textbf{C}.$ Therefore $T(x)\bot T(y)$ in $Y$. Hence $(2)\Rightarrow (1).$

We want to show $(1)\Rightarrow (3).$ In the proof of $(1)\Rightarrow (2),$ we showed that the support of the linear functional $l:C(K_x)\rightarrow \textbf{C}$ is a singleton. That is for a fixed $t\in K_w,$
\begin{equation} \label{eq71}
 l(y)=T(y)(t)=\lambda _t y(s_t)
\end{equation}
for some $s_t\in K_x$ and some $\lambda _t \in \textbf{C}$, for all $y\in C(K_x)=I(x).$ Note that, by $(2),$ we may choose $w=|T(x)|\in Y.$ Given $x_1, x_2 \in X_+$, let $x=x_1 + x_2$. Then for any $t\in K_w$, by~(\ref{eq71}), we have $|T(x_1 + x_2)(t)|=|T(x_1)(t)|+|T(x_2)(t)|$. Since $t$ is arbitrary, in fact, $|T(x_1 + x_2)|=|T(x_1)|+|T(x_2)|$ for all $x_1, x_2 \in X_+$. This yields a well defined  positive operator $(T_l)_r :X_r\rightarrow Y_r$ where $(T_l)_r( x_1-x_2)=|T(x_1)|-|T(x_2)|$ for all $x_1, x_2 \in X_+$. In fact $(T_l)_r$ is a lattice homomorphism. This can be seen easily from~(\ref{eq71}), if we take $x_1\bot x_2$ in above and $x=x_1+x_2$. We can write the complexification of $(T_l)_r$, to obtain $T_l:X\rightarrow Y$ defined by $T_l(z)=(T_l)_r(u)+i(T_l)_r(v) =|T(u^+)|-|T(u^-)|+i(|T(v^+)|-|T(v^-)|)$ when $z=u+iv$. Let $x=|z|$ in~(\ref{eq71}). Then we get $|T_l(z)|=|T(z)|$. Now apply $(2)$ to $|z|\leq |z|$, to get $|T(z)|=|T(|z|)|.$ But $|T(|z|)|=T_l(|z|)$. Hence $T_l$ is a lattice homomorphism.

It remains to see that $T$ factors through $T_l$ unimodularly on the support of the range of $T$. Fix a representation of $\widehat{Y}$ in $C_\infty (Q(Y))$. Consider the ranges of $T$ and $T_l$ as subsets of $\widehat{Y}$ in $C_\infty (Q(Y))$. For each $x\in X_+,$ we have $T_l(x)=|T(x)|.$ Since $\widehat{Y}$ is Dedekind complete, there exists $V_x\in Z(\widehat{Y})$ such that $T(x)=V_xT_l(x)$ and $|V_x|=\chi _{T(x)}.$ (Note that $\chi _{T(x)}$ is also the support of $T_l(x).$) We want to show
\begin{equation*}
\chi _{T(x)}V_y=\chi _{T(y)}V_x.
\end{equation*}
That is $V_x$ and $V_y$ are equal on the intersection of the supports of $T(x)$ and $T(y).$ Since $T_l$ is a positive operator, we have $\chi _{T(x+y)}=max\left\{\chi _{T(x)} ,\chi _{T(y)}\right\}.$ Now
\begin{equation*}
V_{x+y}T_l(x+y)=T(x+y)=T(x)+T(y)=V_xT_l(x)+V_yT_l(y).
\end{equation*}
That is
\begin{equation*}
(V_{x+y}-V_x)T_l(x)=(V_y-V_{x+y})T_l(y).
\end{equation*}
Now multiplying both sides of the above equation by $\chi _{T(x)}\chi _{T(y)}$, we have
\begin{equation} \label{eq72}
  \chi _{T(y)}(V_{x+y}-V_x)T_l(x)=\chi _{T(x)}(V_y-V_{x+y})T_l(y).
\end{equation}
Suppose for some $t\in Q(Y)$ both $T_l(x)(t)$ and $T_l(y)(t)$ are positive. Let $k=\frac {T_l(x)(t)} {T_l(y)(t)} >0.$ Now evaluate~(\ref{eq72}) at $t$ and simplify to obtain $(k+1)V_{x+y}(t)=kV_x(t)+V_y(t).$ Since $|V_{x+y}(t)|=|V_x(t)|=|V_y(t)|=1,$ we have $k+1=|kV_x(t)+V_y(t)|$ and therefore $V_x(t)=V_y(t).$ This proves that $\chi _{T(x)}V_y=\chi _{T(y)}V_x.$

Evidently $\chi _T=sup\left\{\chi _{T(x)}:x\in X_+\right\}.$ The set of $t\in Q(Y)$ such that $T(x)(t)$ is finite and non-zero for some $x\in X_+$ forms a dense open subset of the support of the range of $T$ in $Q(Y).$ For any such $t\in Q(Y)$ and $x\in X_+$ define  $V(t)=V_x(t)$ and whenever $\chi _T(t)=0$ for some $t\in Q(Y)$ define $V(t)=0$. Since $Q(Y)$ is a Stonian compact Hausdorff space and $V$ as defined is continuous and bounded on a dense open subset of $Q(Y)$, we can extend it to a unique $V\in C(Q(Y)).$ Therefore there exists a unique $V\in Z(\widehat{Y})$ such that $T(x)=VT_l(x)$ for all $x\in X_+$ and $|V|=\chi _T.$ This means $T=VT_l$ as claimed and $(1)\Rightarrow (3).$

On the other hand, if $T=VT_l$ as in $(3)$, since $T_l$ and $V$ are both disjointness preserving, we have that $T$ is disjointness preserving. That is $(3)\Rightarrow (1)$ trivially.
\end{proof}

We note the following interesting consequence of the above result.

\begin{corollary} \label{at2}
Suppose $X$ and $Y$ are complex Banach lattices and $T_r:X_r\rightarrow Y_r$ is a lattice homomorphism. Then its complexification $T:X\rightarrow Y,$ defined by $T(x+iy)=T_r(x)+iT_r(y)$ for all $x,y\in X_r,$ is a lattice homomorphism.
\end{corollary}

\begin{proof}
$T$ is disjointness preserving. Namely, suppose $|x+iy|\bot |u+iv|$ for some $x,y,u,v\in X_r$. Since $max\left\{|x|,|y|\right\} \leq |x+iy|$ and $max\left\{|u|,|v|\right\} \leq |u+iv|$, $T_r(x)$ and $T_r(y)$ are disjoint from each of $T_r(u)$ and $T_r(v)$. That is, $|T_r(x)|+|T_r(y)|$ is disjoint from $|T_r(u)|+|T_r(v)|$. Hence $|T(x+iy)|\bot |T(u+iv)|$.

Then by part $(3)$ of Theorem $\ref{at1}$, $T=VT_l$ for some lattice homomorphism $T_l$, where $|V|=\chi _T.$ But for each $x\in X_+$, we have $T_l(x)=|T(x)|=|T_r(x)|=T_r(x)$. This means that $T=T_l,$ and $T$ is a lattice homomorphism.

\end{proof}

Next result is another application of Theorem $\ref{at1}$. In the case of Dedekind complete real Banach lattices it is due to Hart~\cite{Ha}. A result due to Luxemburg and Schep~\cite{LS} and to Lipecki~\cite{Li} shows that a lattice homomorphism on real Banach lattices may be extended to a lattice homomorphism on their Dedekind completions. It allows us to extend Hart's result to arbitrary Banach lattices (see Theorem~\ref{at3} below).

\begin{remark} \label{r26} Indeed, an analog of Hart's result is true even in the case of arbitrary Riesz spaces (see~\cite[Theorem 3.4 and Theorem 6.7]{AAK}), but the corresponding proof in~\cite{AAK} is rather complicated, long, and some details are compressed. For the sake of completeness and the reader's convenience we decided to provide an independent and much simpler proof in the case of Banach lattices, which is sufficient for our purposes in the current paper.
\end{remark}

In the statement of the theorem below, $Y_T$ denotes the closed ideal generated by the range of the operator $T$ in the (complex) Banach lattice $Y.$

\begin{theorem} \label{at3}
Suppose $X,Y$ are Banach lattices and $T:X\rightarrow Y$ is a disjointness preserving operator. Then there exists an algebra homomorphism $\gamma:Z(X)\rightarrow Z(Y_T)$ such that $$T(fx)=\gamma (f)T(x)$$ for all $f\in Z(X)$ and $x\in X.$

\end{theorem}
\begin{proof}
Suppose that $T:X\rightarrow Y$ is a lattice homomorphism and $X,Y$ are Dedekind complete. Let $e\in Z(X)$ be an idempotent and $x\in X_+.$ Then $0\leq T(ex)\leq T(x)$ implies that there exisxts a minimal $\widehat{e} _x\in Z(Y)$ (i.e. $\chi _{T(x)} \widehat{e} _x=\widehat{e} _x$) such that $\widehat{e} _xT(x)=T(ex).$ Similarly there exists a minimal $\widehat{(1-e)} _x\in Z(Y)$ such that $\widehat{(1-e)} _xT(x)=T((1-e)x).$ But, since $T(ex)\bot T((1-e)x)$, we have $\widehat{e} _x+\widehat{(1-e)} _x=\chi _{T(x)}$ and $\widehat{e} _x\bot \widehat{(1-e)} _x.$ That is $\widehat{e} _x$ is a minimal idempotent. If $0\leq y \leq x$, then that $0\leq T(y)\leq T(x),$ implies that $\widehat{e} _y\leq \widehat{e} _x.$ Now define an idempotent $\gamma (e)\in Z(Y)$ by $\gamma (e)=sup\left\{\widehat{e} _x: x\in X_+\right\}$. Then $\chi _{T(x)}\gamma (e)=\widehat{e} _x$ and $T(ex)=\gamma (e)T(x)$ for all $x\in X_+.$ By linearity, we have $T(ex)=\gamma (e)T(x)$ for all $x\in X.$ As defined $\gamma (e)\in Z(Y)$ is unique. So, given idempotents $e_1, e_2\in Z(X)$, that $T(e_1(e_2x))=\gamma (e_1)\gamma (e_2)T(x)$ for all $x\in X$, implies that $\gamma (e_1e_2)=\gamma (e_1)\gamma (e_2).$ Let $sp(\textbf{B}_X)$ denote the linear span of the band projections in $Z(X).$ Then it is immediate that $\gamma :sp(\textbf{B}_X)\rightarrow sp(\textbf{B}_Y)$ is a well-defined algebra homomorphism with $\left\|\gamma\right\|=1$ where $\gamma (1)=\chi _T$ and $\gamma (\sum \lambda _i e_i)=\sum \lambda _i \gamma (e_i)$ for any finite set of $e_i\in \textbf{B}_X$ and $\lambda _i \in \textbf{C}$ with $i=1,2,...n$. Since $Q(X)$ is extremally disconnected, $sp(\textbf{B}_X)$ is dense in $Z(X)=C(Q(X))$ in the sup-norm. Using density and continuity, we extend to an algebra homomorphism $\gamma :Z(X)\rightarrow Z(Y)$ such that $\gamma (1)=\chi _T$ and $$T(fx)=\gamma (f)T(x)$$ for all $f\in Z(X)$ and $x\in X.$

Returning to the general case, suppose $T:X\rightarrow Y$ is a disjointness preserving operator. By Theorem $\ref{at1}$ part (3), we have $T=VT_l$ where $T_l:X\rightarrow Y$ is a lattice homomorphism and $V\in Z(\widehat{Y})$ such that $|V|=\chi _T.$ When restricted to $X_r$, $T_l$ is a (real) lattice homomorphism into $Y_r.$ By the above quoted result of Luxemburg and Schep~\cite{LS} and of Lipecki~\cite{Li}, there is an extension of $(T_l)_{|X_r}$ to $\widehat{T_l}:\widehat{X_r}\rightarrow \widehat{Y_r}$ that is a lattice homomorphism. By Corollary $\ref{at2}$, its complexification $\widehat{T_l}: \widehat{X}\rightarrow \widehat{Y}$ is also a lattice homomorphism. Therefore, by the first part of the proof, there exists an algebra homomorphism $\gamma :Z(\widehat{X})\rightarrow Z(\widehat{Y})$ such that $$\widehat{T_l}(fx)=\widehat{\gamma}(f)\widehat{T_l}(x)$$ for all $f\in Z(\widehat{X})$ and $x\in \widehat{X}.$ Clearly $\widehat{T}:\widehat{X}\rightarrow \widehat{Y}$ defined by $\widehat{T}=V\widehat{T_l}$ is a disjointness preserving operator that extends $T$. Also it satisfies $$\widehat{T}(fx)=\widehat{\gamma}(f)\widehat{T}(x)$$ for all $f\in Z(\widehat{X})$ and $x\in \widehat{X}.$ (Note that V commutes with $\widehat{\gamma}$) Hence if we let $\gamma=\widehat{\gamma}_{|Z(X)}$ and recall $T=\widehat{T}_{|X}$, we almost obtain the statement of the theorem. The exception is that, for each $f\in Z(X)$, we have $\gamma (f)\in Z(\widehat{Y})$ instead of $\gamma (f)\in Z(Y_T).$ But for each $f\in Z(X),$ $\gamma (f)$ acts only on the range of the operator $T$ and leaves it invariant. Evidently, if we show that $\gamma (f)$ also leaves invariant the closed ideal generated by the range of $T$, then we would have $\gamma (f)\in Z(Y_T)$ and the theorem would be proved.

To show this, let us fix $f\in Z(X)$ and $x\in X_+$. Consider the ideal $I(T(x))$ that is generated by $|T(x)|$ in $Y$. Since $|T(x)|=|T_l(x)|=T_l(x)$, we have $I(T(x))=\left\{aT_l(x): a\in Z(Y(T(x))\right\}$. Let us also fix a representation of $\widehat{Y}$ in $C_{\infty}(Q(Y))$. For any $a\in Z(Y(T(x)))$, we define a unique function $\widehat{a}\in C(Q(Y))=Z(\widehat{Y})$ as follows. If $\widehat{T_l (x)}(t)$ is finite and non-zero for some $t\in Q(Y)$ with $\chi _{T(x)}=1$, let $\widehat{a}(t)=\widehat{aT_l (x)}(t)/\widehat{T_l (x)}(t)$. If $\chi_{T(x)}(t)=0$ for some $t\in Q(Y)$, let $\widehat{a}(t)=0$. Hence $\widehat{a}$ is defined as a bounded continuous function on a dense open subset of the stonean compact set Q(Y). So there is a unique extension $\widehat{a}\in C(Q(Y))=Z(\widehat{Y})$. Since $\widehat{aT(x)}=\widehat{a}\widehat{T(x)}$, we have
\begin{equation*}
  \widehat{\gamma(f)(aT(x))} = \widehat{\gamma(f)}\widehat{a}\widehat{T(x)}=\widehat{a}\widehat{\gamma (f)}\widehat{T(x)}=\widehat{aT(fx)}.
\end{equation*}
 That is $\gamma(f)(aT(x))=aT(fx) \in I(T(x))$, since $T(fx) \in I(T(x))$ by part 2 of Theorem~\ref{at1}. Hence, the right hand side of the last equality is in $Y_T$. This means that $\gamma(f)(I(T(x))\subset Y_T$. But since $T_l$ is positive and additive on $X_+$, it is immediate that $Y_T=cl(\cup\left\{I(T(x)):x\in X_+\right\})$. Therefore $\gamma(f)(Y_T)\subset Y_T$. This completes the proof.
\end{proof}

In the rest of this section we will be reviewing the concept of disjointness in Banach $C(K)$-modules and the basic properties of disjointness preserving operators on Banach $C(K)$-modules. The general conditions we assume on Banach $C(K)$-modules are stated in Subsection 2.2 of this article. The only exception is that we will not assume that $m(C(K))$ is weak-operator closed in $L(X)$, since it is not necessary for the concepts we will discuss here.

First we will review the notion of disjointness of elements in a Banach $C(K)$-module. We will define a pair of elements $x,y\in X$ to be \textbf{disjoint}, if there exists some $u\in X$ such that $x,y\in X(u)$ and $x\bot y$ in $X(u)$ when the cyclic space $X(u)$ is represented as a Banach lattice as explained in Subsection 2.2 ( see also~\cite[Lemma 2.4]{KO}). This is equivalent to the definition given
in~\cite[Definition 4.4, p.21]{AAK}. But to show this, we need to introduce the following object defined in~\cite[Definition 4.1, p.20]{AAK}. Namely, for any $x\in X$, we have
\begin{equation*}
 \Delta (x)=cl\left\{fx: f\in C(K),  \  \left\|f\right\| \leq 1\right\}
\end{equation*}
 where 'cl' is norm-closure in $X.$ (As remarked in~\cite{AAK}, the above defined set $\Delta(x)$ represents an analog for Banach $C(K)$-modules of the notion of an order interval in a Banach lattice.) According to Definition 4,4 in~\cite{AAK}, $x,y\in X$ are called disjoint ($xdy$) if $\Delta (x+y)=\Delta (x)+\Delta (y)$ and $\Delta (x)\cap \Delta (y)=\left\{0\right\}.$

First we want to show that if $x,u\in X$ and $x\in X(u)$, then indeed $\Delta (x)$ is the closed order interval generated by $|x|$ in the complex Banach lattice $X(u).$ That is, $y\in X(u)$ and $|y|\leq |x|$ if and only if there exists $\left\{f_n\right\}$ in $C(K)$ with $\left\|f_n\right\|\leq 1$ such that $f_nx\rightarrow y$ in norm. To see this note that $|y|\leq |x|$ implies $y$ is in the ideal generated by $|x|$ in $X(u).$ But the closed ideal generated by $|x|$ in $X(u)$ is $X(x).$ Assume without loss of generality that $\left\|x\right\|=1.$ Then parts (4) and (5) of Lemma 2.4 in~\cite{KO} applied to $X(x)$ imply the existence of the required sequence. The other direction is self-evident.

Suppose that $xdy$ for some $x,y\in X$ as in Definition 4.4 in~\cite{AAK}. Then $x,y\in \Delta (x+y)$ implies that $x,y\in X(u)$ with $u=x+y.$ Also that $\Delta (x)\cap \Delta (y)=\left\{0\right\}$ implies that the order intervals generated $|x|, \ |y|$ in $X(u)$ are disjoint, that is $|x|\wedge |y|=0.$ Hence $x\bot y$ in $X(u).$ Conversely, suppose for some $x,y\in X(u)$ we have $x\bot y$ in the Banach lattice $X(u).$ Then $\Delta (x)\cap \Delta (y)=\left\{0\right\}$ and $|x+y|=|x|+|y|.$ Since $f,g\in C(K)$ are in the unit ball of $C(K)$ implies that $|fx+gy|\leq |x|+|y|=|x+y|$, we obtain $\Delta (x) +\Delta (y)=\Delta (x+y).$

Now we want to show that the definition of a disjointness preserving operator that we give here (Definition 2.6) is equivalent to the one given in~\cite[Definition 4.5, p.21]{AAK}.

Let $X$ be a Banach $C(K_X)$-module and $Y$ be a Banach $C(K_Y)$-module. Let $T:X\rightarrow Y$ be an operator. We say (Definition 2.6) $T$ is a \textbf{disjointness preserving} operator if, for each $x\in X,$ (1) $T(X(x))\subset Y(T(x))$ and (2) $T$ restricted to $X(x)$ is a disjointness preserving operator when $X(x)$ and $Y(T(x))$ are considered as Banach lattices. On the other hand the same operator $T$ is called disjointness preserving in~\cite[Definition 4.5, p.21]{AAK}, if for all $x,y\in X$ with $xdy$, we have $T(x)dT(y)$ in $Y.$ It is easy to see if the operator $T$ satisfies Definition 2.6 above, then it also satisfies Definition 4.5 in~\cite{AAK}. To show this assume that $xdy$ in $X.$ Then $x\bot y$ in $X(x+y).$ Therefore, by Definition 2.6, $T(x)\bot T(y)$ in $Y(T(x+y))$. But $T(x+y)=T(x)+T(y)$, so $T(x)dT(y)$ in $Y$. Hence $T$ satisfies Definition 4.5 in~\cite{AAK}. But proving the converse implication, requires some work.

We need two lemmas from~\cite{AAK}. Initially, given $x\in X$ and $x^*\in X'$ define $x^*\Box x\in C(K_X)'$ by
\begin{equation*}
   x^*\Box x(f)=x^*(fx)
\end{equation*}
 for all $f\in C(K_X)$ (see~\cite[Definition 4.7, p.23]{AAK}). The lemma below is from \cite[Lemma 4.8, part(3), p.23]{AAK}.

\begin{lemma} \label{at4}
Let $x,y\in X.$ Then the following are equivalent.
\begin{enumerate}
\item $y\in X(x)$ and $|y|\leq x$ in the Banach lattice $X(x)$.
\item For each $x^*\in X'$, $\left\|x^*\Box y\right\|\leq \left\|x^*\Box x\right\|$.
\item For each $x^*\in X'$, $|x^*\Box y|\leq |x^*\Box x|$.

\end{enumerate}

\end{lemma}
\begin{proof}
$(1)\Rightarrow (3)$ is almost immediate. Assume (1). Consider the complex Banach lattice $C(K_X)'$ as a module over its center $Z(C(K_X)')=C(K_X)'^{\prime}$. Given any $f\in C(K_X),$ it follows from definitions that $x^*\Box fx=f\cdot(x^*\Box x)$ where, on the right hand side, we consider $f$ acting as a central operator on $x^*\Box x$. Then $|x^*\Box fx|\leq \left\|f\right\||x^*\Box x|$. If $|y|\leq x$ in $X(x)$, there exists a sequence $\left\{f_n\right\}$ in the unit ball of $C(K_X)$ such that $f_nx\rightarrow y$ in norm. (See the discussion on $\Delta (x).$)Then $|x^*\Box y|\leq |x^*\Box x|$ for all $x^*\in X'.$

That $(3)\Rightarrow(2)$ is evident.

To show $(2)\Rightarrow (1)$ : suppose $y\notin \Delta (x)$. Since $\Delta (x)$ is an absolutely convex closed set in $X$, by standard separation theorems on real Banach spaces there exists a bounded real linear functional $x^*_r$ on $X$ such that
\begin{equation*}
 sup\left\{|x^*_r(z)|:z\in \Delta (x)\right\}<1<x^*(y).
\end{equation*}
 Then, as standard, define $x^*(z)=x^*_r(z)-ix^*_r(iz)$ for each $z\in X.$ We know that $x^*\in X'$ and $\left\|x^*\right\|=\left\|x^*_r\right\|$. For each $f$ in the unit ball of $C(K_X)$ and for some real $\theta$, we have
 \begin{equation*}
 |x^*(fx)|=e^{i\theta}x^*(fx)=x^*(e^{i\theta}fx)=x^*_r(e^{i\theta}fx)<1<x^*_r(y)\leq |x^*(y)|.
 \end{equation*}
 (Note that $e^{i\theta}fx\in \Delta (x).$) Hence $\left\|x^*\Box x\right\|<\left\|x^*\Box y\right\|$ and $(2)\Rightarrow (1)$.

\end{proof}

The next lemma is also from\cite[Lemma 4.8, part(4), p.23]{AAK}.

\begin{lemma} \label{at5}
Let $x,y\in X$. Then the following are equivalent.
\begin{enumerate}
\item $x,y\in X(x+y)$ and $x\bot y$ in $X(x+y).$
\item For each $x^*\in X'$, $x^*\Box x\bot x^*\Box y$ in $C(K_X)'.$
\end{enumerate}

\end{lemma}

\begin{proof}
To see $(1)\Rightarrow(2)$: Note $x\bot y$ implies $|x|\leq |x+\lambda y|$ for all $\lambda\in \textbf{C}$ in $X(x+y)$. Then, by Lemma \ref{at4} part(3), for each $x^*\in X'$, we have $|x^*\Box x|\leq |x^*\Box (x+\lambda y)|$ for all $\lambda\in \textbf{C}$. Hence $x^*\Box x\bot x^*\Box y$ in $C(K_X)'.$

To see $(2)\Rightarrow (1)$: Note that $(2)$ implies, for each $x^*\in X'$, $|x^*\Box x|\leq |x^*\Box (x+\lambda y)|$ for all $\lambda\in \textbf{C}$. When $\lambda =1$, by Lemma \ref{at4} part(1), we have that $x$ is, and therefore also $y$ is in $X(x+y)$. Then, the full inequality, again by Lemma \ref{at4} part(1), yields $|x|\leq |x+\lambda y|$ for all $\lambda \in \textbf{C}.$ Hence $x\bot y$ in $X(x+y).$ This completes the proof.

\end{proof}

Our next result is the claimed equivalence between Definition 2.6 of this paper and Definition 4.5 in \cite{AAK}. The interested reader should see~\cite[Theorem 6.6, p.37]{AAK} for the idea behind the proof. Given an operator $T:X\rightarrow Y$ between Banach modules $X,Y$ by the notation $T_{|X(x)}$, we denote the restriction of $T$ to the cyclic subspace $X(x)$ of $X$.

\begin{theorem} \label{at6}
Suppose $X,Y$ are Banach $C(K_X)$ and Banach $C(K_Y)$ modules respectively. Let $T:X\rightarrow Y$ be an operator. Then the following are equivalent.
\begin{enumerate}
\item For any $x,y\in X,$ $xdy$ in $X$ implies $T(x)dT(y)$ in $Y.$
\item For each $x\in X,$ $T(X(x))\subset Y(T(x))$ and $T_{|X(x)}:X(x)\rightarrow Y(T(x))$ is a disjointness preserving operator.
\end{enumerate}
\end{theorem}

\begin{proof}
We already showed $(2)\Rightarrow(1)$ when we introduced the two definitions. To see $(1)\Rightarrow (2)$: Fix an $x\in X$ and take any $y^*\in Y'.$ Define an operator $U:X(x)\rightarrow C(K_Y)'$ by
\begin{equation*}
  U(z)=y^*\Box T(z),
\end{equation*}
 for all $z\in X(x).$ If $y\bot z$ in $X(x)$, then $y\bot z$ in $X(y+z)$($\subset X(x)$). By $(1)$, we have $T(y)\bot T(z)$ in $Y(T(y)+T(z))$. Then Lemma \ref{at5} part(2) gives $y^*\Box T(y) \bot y^*\Box T(z)$ in $C(K_Y)'.$ So $U$ is a disjointness preserving operator between the given Banach lattices. By Theorem \ref{at1} part(2), we have $|y|\leq |z|$ in $X(x)$ implies $|y^*\Box T(y)|\leq |y^*\Box T(z)|$ in $C(K_Y)'.$ Since $y^*\in Y'$ is arbitrary, by Lemma \ref{at4} part(1) we have
 \begin{equation*}
   T(y)\in Y(T(z)) \; \text{and} \:  |T(y)|\leq |T(z)|.
 \end{equation*}
 When we take $z=x$ in the above inequality, we get $T(X(x))\subset Y(T(x)).$ On the other hand, the displayed inequality shows that the operator $T_{|X(x)}:X(x)\rightarrow Y(T(x))$ satisfies Theorem \ref{at1} part(2). Therefore $T_{|X(x)}$ is a disjointness preserving operator between the given Banach lattices. So $(1)\Rightarrow (2)$
\end{proof}

\end{document}